\makeatletter
\ifdefined\@kernel@refstepcounter
  \let\KLsquare@refstepcounter\refstepcounter
  \AddToHook{package/hyperref/before}{\let\refstepcounter\KLsquare@refstepcounter
  }
  \AddToHook{package/cleveref/after}{\let\KLsquare@refstepcounter@optarg\refstepcounter@optarg
    \def\refstepcounter@optarg[#1]#2{\crefalias{#2}{#1}\KLsquare@refstepcounter@optarg[#1]{#2}}}
  \AddToHook{begindocument/end}{\def\KLsquare@label@first#1#2\@nil#3{\def#3{#1}}\def\KLsquare@label@second#1#2#3\@nil#4{\def#4{#2}}\def\cref@getref#1#2{\expandafter\expandafter\expandafter\KLsquare@label@first
        \csname r@#1@cref\endcsname\@nil#2}\def\cpageref@getref#1#2{\expandafter\expandafter\expandafter\KLsquare@label@second
        \csname r@#1@cref\endcsname\@nil#2}}
\fi
\makeatother
\documentclass[hidelinks,onefignum,onetabnum]{siamart220329}

\usepackage{lipsum}
\usepackage{epstopdf}
\ifpdf
  \DeclareGraphicsExtensions{.eps,.pdf,.png,.jpg}
\else
  \DeclareGraphicsExtensions{.eps}
\fi
\usepackage{amsfonts}
\usepackage{mathrsfs}
\usepackage{graphicx}
\usepackage{color}
\graphicspath{{figure/}}
\usepackage{indentfirst,latexsym,bm}
\usepackage{algorithm}
\usepackage{algpseudocode}
\usepackage{amsmath}
\usepackage{extarrows}
\usepackage{hyperref}
\definecolor{orangep}{RGB}{255,128,0}
\definecolor{bluep}{RGB}{0,0,255}
\hypersetup{
    colorlinks=true,
    linkcolor=orangep,
    filecolor=magenta,
    urlcolor=cyan,
    pdftitle={Overleaf Example},
    pdfpagemode=FullScreen,
}

\DeclareMathOperator*{\argmin}{Arg\,min}
\newcommand{\R}{\mathbb{R}}
\newcommand{\dom}{\mathrm{dom}\,}

\newcommand{\dist}{\mathrm{dist}}

\newcommand{\aff}{\mathrm{aff}}
\newcommand{\parr}{\mathrm{par}}

\newcommand{\ri}{\mathrm{ri}}

\newcommand{\spa}{\mathrm{span}}

\newcommand{\bB}{\mathbb{B}}
\newcommand{\cS}{\mathcal{S}}
\newcommand{\cX}{\mathcal{X}}
\newcommand{\cY}{\mathcal{Y}}

\newcommand{\n}{\mathbb{N}}
\newcommand{\cD}{\mathcal{D}}

\newcommand{\cM}{\mathcal{M}}
\newcommand{\cA}{\mathcal{A}}
\newcommand{\cB}{\mathcal{B}}

\newcommand{\cL}{\mathcal{L}}

\newcommand{\cP}{\mathcal{P}}

\newcommand{\sgn}{\mathrm{sgn}}
\newcommand{\cR}{\mathcal{R}}

\makeatletter
\newcommand{\refcheckize}[1]{\expandafter\let\csname
	@@\string#1\endcsname#1\expandafter\DeclareRobustCommand\csname
	relax\string#1\endcsname[1]{\csname
		@@\string#1\endcsname{##1}\wrtusdrf{##1}}\expandafter\let\expandafter#1\csname
	relax\string#1\endcsname
}
\makeatother
\usepackage{enumerate}
\usepackage{cite} 

\usepackage{subfigure}
\usepackage{amssymb}
\usepackage{microtype}

\makeatletter
\renewtheoremstyle{plain}{\item[\hskip\labelsep\hskip\parindent \theorem@headerfont ##1\ \textup{##2\theorem@separator}]}{\item[\hskip\labelsep\hskip\parindent \theorem@headerfont ##1\ \textup{##2}]\textup{\let\cite@adjust\@empty(\ignorespaces##3)\theorem@separator}\ }
\makeatother

\newsiamremark{hypothesis}{Hypothesis}
\crefname{hypothesis}{Hypothesis}{Hypotheses}
\newsiamthm{claim}{Claim}

\newsiamthm{thm}{Theorem}
\newsiamthm{prop}{Proposition}
\newsiamthm{coro}{Corollary}
\newsiamthm{fact}{Fact}
\newsiamthm{lem}{Lemma}
\newsiamthm{assumption}{Assumption}
\newsiamthm{defn}{Definition}

\newsiamremark{remark}{Remark}
\newsiamremark{exmp}{Example}

\headers{Variational Properties of Square Transformation}{W. Ouyang}

\title{Variational Properties and KL Exponent of Square-Transformed Polyhedral Composite Optimization}

\author{Wenqing Ouyang\thanks{(\url{wenqing.ouyang@tse-fr.eu}), IEOR, Columbia University, New York, USA. Research supported by NSF EPCN grant 2023032.}}

\usepackage{amsopn}
\DeclareMathOperator{\diag}{diag}
\DeclareMathOperator{\grad}{grad}

\allowdisplaybreaks[2]

\def\N{\mathbb{N}}
\usepackage{enumitem}
\crefname{enumi}{}{}
\ifpdf
\hypersetup{
  pdftitle={Variational Properties and KL Exponent of Square-Transformed Polyhedral Composite Optimization},
  pdfauthor={W. Ouyang}
}
\fi

\begin{document}

\maketitle

\begin{abstract}
  We consider one of the most common reparameterization techniques, the square transformation. Assuming the original objective function is the sum of a smooth function and a polyhedral function, we study the variational properties of the objective function after reparameterization. In particular, we first study the minimal norm of the subdifferential of the reparameterized objective function. Second, we compute the second subderivative of the reparameterized objective function on a linear subspace, which {allows us to fully characterize} the subclass of stationary points of the reparameterized objective function that correspond to stationary points of the original objective function. Finally, utilizing the representation of the minimal norm of the subdifferential, we show that the Kurdyka-\L ojasiewicz (KL) exponent of the reparameterized function can be deduced from that of the original function. {Applications to concrete optimization models are discussed.}
\end{abstract}



\section{Introduction and Background}
\label{sec1}
Introducing square variables to eliminate nonnegative constraints is a simple yet effective reformulation technique. Given an optimization problem as follows:
\begin{align}
  \label{nonnegative_cons_prob}
  \min_{x\in\R^n} f(x),\quad s.t.~x\geq 0,
\end{align}
where $f\in C^2(\R^n)$. By introducing $y$ such that $x=y^2$, we can transform \cref{nonnegative_cons_prob} to the following unconstrained problem:
\begin{align}
  \label{un_sq_prob}
  \min_{y\in \R^n} f(y^2).
\end{align}
This kind of reparameterization is also called Hadamard parameterization \cite{li2023simplex}, which can be traced back to the textbook \cite[p. 268--269]{gill1981practical}. Transforming \cref{nonnegative_cons_prob} to \cref{un_sq_prob} makes algorithms designed for unconstrained optimization applicable, such as Newton and quasi-Newton methods. In general, if there are additional constraints in \cref{nonnegative_cons_prob}, the square transformation can also eliminate the nonnegativity constraints, which can be exploited in algorithmic design \cite{li2023simplex,tang2024optimization,liu2025linear,chou2025get}. For example, in \cite{li2023simplex}, the authors consider the simplex constrained optimization problem:
\begin{equation}
  \label{simplex_cons_prob}
  \min_{x\in \R^n}f(x),\quad s.t.~\sum_{i=1}^n x_i=1,\quad x\geq 0,
\end{equation}
where $f\in C^2(\R^n)$. By setting $x=y^2$, this problem can be transformed to 
\begin{align}
  \label{ball_cons_prob}
  \min_{y\in \R^n} f(y^2),\quad s.t.~\sum_{i=1}^n y_i^2=1,
\end{align}
{which is a smooth optimization problem on a Riemannian manifold.} More surprisingly, combined with the implicit bias of gradient descent \cite{li2021implicit,tibshirani2021equivalences,zhao2022high}, in \cite{vaskevicius2019implicit,zhao2022high} the square transformation is used to solve the basis pursuit problem \cite{chen2001atomic,donoho2006compressed,chen1996basis}.  
Further applications to nonlinear programming problems can be found in \cite{ding2023squared,fukuda2017note,Bertsekas99}, where the square transformation is used to eliminate the nonnegativity constraints on the slack variables.

In this paper, we consider the following optimization problem
  \begin{equation}
    \label{prob1}
    \min_{x\in\R^n} \phi(x):=f(x)+g(x)  \end{equation}
where we assume $f\in C^2(\R^n)$ and $g:\R^n\to \R\cup\{\infty\}$ is a proper polyhedral function with $\dom g\subseteq \R_+^n$. We also consider the square transformation applied to \cref{prob1}:
\begin{equation}
  \label{sqprob1}
  \min_{y\in \R^n} \Phi(y):=f(y^2)+g(y^2).
\end{equation}
This includes the models in \cite{li2023simplex,tang2024optimization,chou2025get} as special cases. {We are interested in the variational properties of \cref{sqprob1}, specifically, how to deduce the KL exponent at stationary points of \cref{sqprob1} from that of \cref{prob1}, and how to characterize a subclass of stationary points of \cref{sqprob1} that correspond to stationary points of \cref{prob1} via the square transformation $T:y\mapsto y^2$. Moreover, we use $DT$ to denote the Jacobian of $T$.}

There are existing works focusing on the correspondence between the stationary points of \cref{prob1} and \cref{sqprob1} in some special cases. In \cite[Theorem 2.3]{ding2023squared}, it is shown that if $x=y^2$ is a stationary point of \cref{nonnegative_cons_prob}, then 
\[   \forall w\in \R^n,~w_{\cA\cup \cD}=0,\quad w^\top \nabla^2 f(x)w\geq 0,              \]
is equivalent to {$y$ being a second-order stationary point of \cref{un_sq_prob}},
where $\cA=\{i\in [n]:~x_i=0,~\nabla f(x)_i>0\}$ and $\cD=\{i\in [n]:~x_i=\nabla f(x)_i=0\}$. Furthermore, in \cite[Theorem 1]{li2023simplex}, the authors prove that $y$ is a second-order Karush-Kuhn-Tucker (KKT) point of \cref{ball_cons_prob} if and only if $y^2$ is a second-order KKT point of \cref{simplex_cons_prob}. Moreover, in \cite{tang2024optimization}, by assuming that $g$ in \cref{prob1} is an indicator function of a kind of polyhedral set\footnote{This set should have the form $\{x\in \R^n:~Ax=b,x\geq 0\}$ for some $A\in\R^{m\times n}$ and $b\in \R^m$. {Let us note that for a general polyhedral function, one can use the standard linear programming reduction technique to reduce \cref{prob1} to the form in this case. However, it is not clear how the variational properties of the objective function would change after this reformulation.}}, the authors show that if $y$ is a second-order stationary point of \cref{sqprob1}, then $y^2$ is a first-order stationary point of \cref{prob1}. This kind of implication ``second-order stationary point of \cref{sqprob1} $\implies$ first-order stationary point of \cref{prob1}" is discussed in \cite[Section 3.2]{levin2025effect} under the setting that $g$ is an indicator function of a smooth manifold {$\cM$, and $T^{-1}\cM$ is also a smooth manifold.} However, all these previous results are highly specialized for the structure of constraints. 

As for the KL exponent of \cref{sqprob1}, the only relevant work we are aware of is \cite{ouyang2024kurdyka}, where the authors calculate the KL exponent of the Hadamard difference parameterization of $\ell_1$ norm regularized optimization problems. This work borrows some techniques from \cite{ouyang2024kurdyka} to deal with the case where $f$ is convex and the strict complementarity condition does not hold. Let us note that the essential difference with \cite{ouyang2024kurdyka} is that in this paper the lifting problem \cref{sqprob1} is not necessarily smooth, which makes the KL exponent calculation more elusive. 

In summary, our contributions include:
\begin{itemize}
  \item By analyzing the subdifferential of $\Phi$, we establish that 
  $$ \forall y\in \dom\partial\Phi,~~\dist(0,\partial \Phi(y))=\dist(0,2y\circ (\nabla f(y^2)+\partial g(y^2))),$$ which forms the basis of calculating the KL exponent of $\Phi$.
  \item We establish the formula of the second subderivative of {$g\circ T$ at $y$} in a certain linear subspace, which allows us to characterize the subclass of stationary points of $\Phi$ in \cref{sqprob1} which correspond to stationary points of $\phi$ in \cref{prob1}. Using the second subderivative helps us get rid of structural assumptions on $g$ in \cref{prob1} like in the literature \cite{ding2023squared,tang2024optimization,levin2025effect}. 
  \item We deduce the KL exponent of $\Phi$ from that of $\phi$. Specifically, we show that if $\phi$ satisfies the KL property at $\bar x$ with exponent $\alpha$ and the strict complementarity condition holds at $\bar x$, then $\Phi$ satisfies the KL property with exponent {$\max\{\alpha,1/2\}$} at $\bar y$, where $\bar x=\bar y^2$. If the strict complementarity condition fails and we assume convexity of $\phi$ and a certain kind of error bound condition with parameter $\gamma\in (0,1]$, then $\Phi$ satisfies the KL property at $\bar y$ with exponent $(1+\beta)/2$, where $\beta=1-\gamma(1-\alpha)$. Similar results are obtained for another kind of Hadamard parameterization problems in \cite{ouyang2024kurdyka}.
  \item {We discuss the connection between the Hadamard parameterization $x=u\circ v$ and the square transformation, and show that our result recovers the KL exponent of the Hadamard parameterization in \cite{ouyang2024kurdyka} as a special case. Moreover, our result can be directly applied to the case where the original problem is not $\ell_1$-norm regularized.}
  \item {We also discuss the application of our results to the convergence rate analysis of algorithms designed for \cref{sqprob1}, and show that the convergence rate can be deduced from the KL exponent of $\Phi$ which can be calculated by our results.}
\end{itemize}
The major technical difficulty lies in the fact that the square transformation {$T$} does not have a surjective Jacobian matrix, which makes the calculus rules in {\cref{fact_chainrule}} inapplicable, and hence we cannot fully characterize $\partial\Phi$, while knowing $\dist(0,\partial\Phi)$ is essential to calculate the KL exponent of $\Phi$. The calculation of the second subderivative also suffers from the same issue. Notice that $T$ can even have zero Jacobian, which means that the constraint qualifications used in \cite[{Theorem 13.14}]{rockafellar2009variational} can fail. Let us note that this kind of constraint qualification is needed for the second-subderivative calculus rules in many other cases, see, e.g. \cite{ouyang2023variational1,mohammadi2021parabolic,benko2023second}. By carefully inspecting the relationship between the subdifferential, the regular subdifferential and the subderivative, we establish the formula of $\dist(0,\partial\Phi)$ without fully characterizing $\partial\Phi$. For the second subderivative, we utilize the fact that polyhedral functions are locally support functions of polyhedral sets {(see \cref{polye})}, {and the geometry of polyhedral sets,} which is then used to characterize the subclass of stationary points of $\Phi$ which correspond to stationary points of $\phi$. 

As for the KL exponent, the major difference with \cite{ouyang2024kurdyka} is that {the model considered in \cite{ouyang2024kurdyka} is a $\ell_1$-regularized function. The result in \cite{ouyang2024kurdyka} is not applicable to many other cases. From a technical perspective, the reparameterized function in \cite{ouyang2024kurdyka} is smooth, and} here the reparameterized function $\Phi$ in \cref{sqprob1} is not necessarily differentiable, or differentiable on a smooth manifold. {This technical difficulty is resolved by using the fruitful geometric properties of polyhedral sets from convex geometry.}   

The rest of the paper is organized as follows. \cref{sec:notations} introduces basic {terminology} and notation used in this paper, along with some preliminary results. In \cref{sec:foa}, we analyze the property of $\partial\Phi$ for $\Phi$ in \cref{sqprob1}, and establish the formula of $\dist(0,\partial\Phi)$. In \cref{sec:soa}, we analyze the second subderivative of $\Phi$ in \cref{sqprob1}, and then characterize the subclass of stationary points of $\Phi$ which correspond to stationary points of $\phi$ in \cref{prob1} using the second subderivative. In \cref{sec:KLstrict}, we deduce the KL exponent of $\Phi$ from that of $\phi$ under {the strict complementarity condition}. {In} \cref{sec:KLnonstrict}, we deduce the KL exponent of $\Phi$ from that of $\phi$ without assuming the strict complementarity condition, but with additional error bound assumption and convexity assumption. {Finally, \cref{sec:app} connects the square transformation with Hadamard parametrizations and illustrates how the resulting KL exponents can be used in the convergence analysis for a hybrid first-order algorithm.}  
\section{Notation and preliminary results}
\label{sec:notations}
In this paper, we use $[r]$ to denote the set $\{1,\dots,r\}$ for nonnegative integer $r$. For vector $x\in \R^n$, we use $x^2$ to denote the componentwise square of $x$, i.e., $(x^2)_i=x_i^2$ for all $i\in [n]$. For vectors $x,y\in \R^n$, we denote the componentwise {product and division of $x$ by $y$ as $x\circ y$ and $x/y$, i.e., $(x\circ y)_i=x_iy_i$ and $(x/y)_i=x_i/y_i$ for all $i\in [n]$}. {Let $\mathbf{1}\in\R^n$ denote the all-ones vector.} {The open ball centered at $x$ with radius $r$ is denoted by $\bB_r(x)$, and if $x=0$ we will write $\bB_r:=\bB_r(0)$.} For an index set $I\subseteq [n]$ and a vector $x\in \R^n$, the notation $x_I\in \R^{|I|}$ denotes the vector consisting {of the components} of $x$ whose indices are in $I$. The notation $I^c$ denotes the complement of the set $I$. For a set $S\subseteq \R^n$, the tangent cone and normal cone of $S$ at $x\in S$ are denoted by $T_S(x)$ and $N_S(x)$, respectively; the projection operator onto the set $S$ is denoted by $\Pi_S$; the distance from a point $x$ to $S$ is denoted by {$\dist(x,S)$}; we denote the indicator function of $S$ by $\iota_S$, and the support function of $S$ by $\sigma_S$; the affine hull of $S$ is denoted by $\aff(S)$, which is defined as the minimal affine set containing $S$; {the linear subspace that is parallel to $S$ is denoted by $\parr (S):=\aff(S-\lambda)$ for any $\lambda\in S$; the orthogonal complement of $S$ is denoted by $S^\perp$; the span of $S$, i.e., the minimal linear subspace that contains $S$ is denoted by $\spa(S)$;} the relative interior of $S$ is denoted by $\ri(S)$, which is the interior of $S$ with respect {to} the topology of $\aff(S)$. For a matrix $M\in \R^{m\times n}$, its Moore-Penrose pseudoinverse is denoted by $M^\dagger$; its spectral norm is denoted by $\|M\|_2$; {its range is denoted by $\cR(M)$}.

For an extended real-valued function $f$, its domain is defined as $\dom f:=\{x\in \R^n:~f(x)<\infty\}$. {Its epigraph is defined as $\mathrm{epi}(f):=\{(x,t)\in\R^{n}\times \R:~t\geq f(x)\}$, and $f$ is said to be polyhedral if $\mathrm{epi}(f)$ is a polyhedral set.} An extended real-valued function $f$ is said to be lower semicontinuous{ (lsc)} at $\bar x$ if {$\liminf_{x\to \bar x}f(x)\geq f(\bar x)$}, and it is lsc if this holds for all $x\in \R^n$; $f$ is said to be proper if $\dom f\neq\emptyset$ and $f>-\infty$. {For a set-valued mapping $F:\R^n\rightrightarrows \R^m$, its domain is defined as $\dom F:=\{x\in \R^n:~F(x)\neq \emptyset\}$. Let us note that a polyhedral function is always lsc and convex, while not necessarily proper.} 

We next introduce several notations from variational analysis. Below is the definition {of subdifferentials}.

\begin{defn}[{\cite[Definition 8.3]{rockafellar2009variational}}]
    Consider $f:\R^n\to \R\cup\{\pm\infty\}$ and a point $\bar x\in \R^n$ with $f(\bar x)$ finite. For a vector $v\in\R^n$, one says that
    \begin{itemize}
        \item[(a)] $v$ is a regular subgradient of $f$ at {$\bar x$}, written $v\in \hat\partial f(\bar x)$, if
        \[   f(x) \geq f(\bar x)+\langle v,x-\bar x \rangle+ o(\|x-\bar x\|),        \]
        \item[(b)] $v$ is a subgradient of $f$ at $\bar x$, written $v\in\partial f(\bar x)$, if there are sequences $x^k\to\bar x$ and $v^k\to v$ with {$v^k\in\hat\partial f(x^k)$} and $f(x^k)\to f(\bar x)$.
    \end{itemize}
\end{defn}
Here are some useful facts about the calculus rules of subdifferentials.
{
  \begin{fact}[{Convexity of the regular subdifferential, \cite[Theorem 8.6]{rockafellar2009variational}}]
    \label{fact_convex_hatpartial}
    For any extended real-valued function $f$ and any $\bar x$ with $f(\bar x)$ finite, the regular subdifferential $\hat\partial f(\bar x)$ is a closed convex set. 
  \end{fact}
   \begin{fact}[{Subdifferential of convex functions, \cite[Proposition 8.12]{rockafellar2009variational}}]
      \label{fact_convex_subdiff}
      Let $f:\R^n\to \R\cup\{\infty\}$ be a proper convex function. Then for any $\bar x\in \dom f$, we have 
      \[ \hat\partial f(\bar x)=\partial f(\bar x)=\{v\in\R^n:\ f(x)\geq f(\bar x)+\langle v,x-\bar x\rangle\text{ for all }x\in\R^n\}. \]
   \end{fact}
}
\begin{fact}[{Sum rule of subdifferentials, \cite[Exercise 8.8(c)]{rockafellar2009variational}}]
  \label{fact_sumrule}
   Let $f,g$ be both extended real-valued functions on $\R^n$. Suppose that $f$ is continuously differentiable around $\bar x$ and $g(\bar x)$ is finite. Then for $\phi=f+g$ we have 
   \[   \partial \phi(\bar x)=\nabla f(\bar x)+\partial g(\bar x),~\hat\partial \phi(\bar x)=\nabla f(\bar x)+\hat\partial g(\bar x).       \]
\end{fact}
\begin{fact}[{Chain rule of the subdifferentials, \cite[Theorem 10.6 {and Exercise 10.7}]{rockafellar2009variational}}]
  \label{fact_chainrule}
  Let $f:\R^m\to \R\cup\{\pm\infty\}$ be an extended real-valued function and let $F:\R^n\to \R^m$ be a continuously differentiable mapping. Define the composition function $\phi=f\circ F$. Then for any $\bar x\in \dom \phi$, we have 
  \[   (DF(\bar x))^\top \hat\partial f(F(\bar x))\subseteq \hat\partial \phi(\bar x).        \]
  If in addition, $DF(\bar x)$ is surjective, then the above inclusion holds with equality, and also holds for the subdifferential $\partial \phi$ in place of $\hat\partial \phi$.
\end{fact}

Below is the definition of second subderivative.
\begin{definition}[{\cite[Definition 13.3]{rockafellar2009variational}}]
  \label{def_second_subderivative}
  Let $f:\R^n\to \R\cup\{\pm\infty\}$ be an extended real-valued function, and let $\bar x\in \dom f$ and {$\bar v\in \R^n$}. Then the second subderivative of $f$ at $\bar x$ for $\bar v$, denoted by $d^2f(\bar x|\bar v):\R^n\to \R\cup\{\pm\infty\}$, is defined as:
  \[  \forall w\in \R^n,\quad d^2f(\bar x|\bar v)(w):=\liminf_{t\downarrow 0,~w'\to w}\frac{f(\bar x+tw')-f(\bar x)-t\langle \bar v,w'\rangle}{t^2/2}.        \]
\end{definition}
{
For a convex function, the subgradient inequality gives the following property of the second subderivative.
\begin{fact}[{Nonnegativity of Second subderivative of convex functions, \cite[Theorem 13.24(a)]{rockafellar2009variational}}]
  \label{fact_convex_second_subderivative}
  Let $h:\R^n\to\R\cup\{\infty\}$ be a proper convex function, $\bar x\in\dom h$, and $\bar v\in\partial h(\bar x)$. Then
  $d^2h(\bar x|\bar v)(w)\geq 0$ for all $w\in\R^n$.
\end{fact}
 
}

{
The next sum rule for second subderivatives is a direct consequence of the definition, which will be useful in our analysis of the second subderivative of the square transformation.
\begin{prop}[{Sum rule for second subderivatives, \cite[Exercise 13.18]{rockafellar2009variational}}]
  \label{prop_sum_second_subderivatives}
  Let $f:\R^n\to\R$ be $C^2$, and let $g:\R^n\to\R\cup\{\pm\infty\}$ satisfy $g(\bar x)\in\R$ for $\bar x\in\R^n$. For any $\bar v,w\in\R^n$,
  \[
    d^2(f+g)(\bar x|\bar v+\nabla f(\bar x))(w)=\langle w,\nabla^2 f(\bar x)w\rangle+d^2g(\bar x|\bar v)(w).
  \]
\end{prop}
\begin{proof}
  By the definition of second subderivative, we have
  \begin{align*}
    &d^2(f+g)(\bar x|\bar v+\nabla f(\bar x))(w)\\
    &=\liminf_{t\downarrow 0,~\tilde w\to w}\frac{f(\bar x+t\tilde w)+g(\bar x+t\tilde w)-f(\bar x)-g(\bar x)-t\langle \bar v+\nabla f(\bar x),\tilde w\rangle}{t^2/2}\\
    &=\langle w,\nabla^2 f(\bar x)w\rangle +\liminf_{t\downarrow 0,~\tilde w\to w}\frac{g(\bar x+t\tilde w)-g(\bar x)-t\langle \bar v,\tilde w\rangle}{t^2/2} \\
    &=\langle w,\nabla^2 f(\bar x)w\rangle+d^2g(\bar x|\bar v)(w).
  \end{align*}
\end{proof}
}

Next, we introduce the definition of KL property.
\begin{defn}[{Kurdyka-\L ojasiewicz property and exponent}]
  \label{defkl} We say that a proper closed function $h:\R^n\rightarrow \mathbb{R} \cup\{\infty\}$ satisfies the KL property at $\bar{x} \in \operatorname{dom} \partial h$ if there are $a \in(0, \infty]$, a neighborhood $V$ of $\bar{x}$ and a continuous concave function $\varphi:[0, {a}) \rightarrow[0, \infty)$ with $\varphi(0)=0$ such that
  
  {\rm (i)} $\varphi$ is continuously differentiable on $(0, {a})$ with $\varphi^{\prime}>0$ on $(0, {a})$;
  
  {\rm (ii)} For any $x \in V$ with $0<h(x)-h(\bar{x})<a$, it holds that
  \begin{equation}
      \label{KLineq}
      \varphi^{\prime}(h(x)-h(\bar{x})) \operatorname{dist}(0, \partial h(x)) \geq 1.
  \end{equation}
  {If $h$ satisfies the KL property at $\bar{x} \in \operatorname{dom} \partial h$ and the function $\varphi$ in \eqref{KLineq} can be chosen as $\varphi(s)=\bar{c} s^{1-\alpha}$ for some $\bar{c}>0$ and $\alpha \in[0,1)$, then we say that $h$ satisfies the KL property at $\bar{x}$ with exponent $\alpha$.}
\end{defn}
{
Below is the definition of stationary point and a special kind of stationary point satisfying the strict complementarity condition.
\begin{defn}[Stationary {points}]
  Let $h:\R^n\to \mathbb{R}\cup\{\infty\}$ be a proper closed function. A point $\bar x\in \dom \partial h$ is called a stationary point of $h$ if $     0\in \partial h(\bar x).$ {It is said to be a second-order stationary point of $h$ if, in addition, $d^2h(\bar x|0)\geq 0$.}
\end{defn}
\begin{defn}[Strict complementarity]
  Let $\phi:\R^n\to \mathbb{R}\cup\{\infty\}$ be a proper closed function, and let $\bar x$ be a stationary point of $\phi$. We say that the strict complementarity condition holds at $\bar x$ if $
    0\in \ri(\partial \phi(\bar x)).$
\end{defn}
}
\begin{remark}
  \label{remarkkl}
   If $h$ is continuous on $\dom h$, then \eqref{KLineq} holds for all $x\in V\cap \dom h$ with $h(x)-h(\bar x)>0$ by reducing $V$ if necessary, since the condition $h(x)-h(\bar x)<a$ holds naturally if $V$ is sufficiently small.
\end{remark}

We conclude this section with several preliminary results that would be useful in the subsequent sections. The first result is about {several useful properties of} a polyhedral function.

{
\begin{fact}[Properties of Polyhedral Functions]
  \label{prop_polyhedral}
  Let $g:\R^n\to \R\cup\{\infty\}$ be a proper polyhedral function. Then the following statements hold:
  \begin{enumerate}[label=\rm{(\alph*)},ref=\cref{prop_polyhedral}~\rm{(\alph*)}]
    \item \label{polya} $\dom g=\dom \partial g$, and for all $x\in \dom g$, $\partial g(x)$ is a nonempty polyhedral set.
    \item  \label{polyd} For all $x\in \dom g$, there exists a neighborhood $V$ of $x$ such that $\partial g(y)\subseteq \partial g(x)$ for all $y\in V$.
    \item \label{polye} For all $x\in \dom g$, there exists a neighborhood $V$ of $x$ such that $g(y)-g(x)=\sigma_{\partial g(x)}(y-x)$ for all $y\in V$.
    \item \label{polyf} For two proper polyhedral functions $g_1,g_2$, we have $\partial(g_1+g_2)=\partial g_1+\partial g_2$.
    \item \label{polyg} There exists a constant $L$ such that for all $x,y\in \dom g$, it holds that $|g(x)-g(y)|\leq L\|x-y\|$.
    \item \label{polyh}{Let $A\in \R^{n\times r}$ and $h:=g\circ A$. Then $\partial h(x)=A^\top \partial g(Ax)$ for all $x\in \R^r$.}
    \item \label{polyi} {Let $A\in \R^{n\times r}$ and $C\subseteq \R^r$ be a polyhedral set. Then $AC$ is a polyhedral set, and in particular it is closed.}
  \end{enumerate}
\end{fact}
\begin{proof}
  (a) follows from \cite[Theorem 23.10]{rockafellar1970convex}. 
  (b) is stated in \cite[Proposition 3.3]{mordukhovich2016generalized}. (c) can be found in \cite[equation (4.2.7)]{facchinei2007finite}, see also \cite[Appendix B]{ouyang2024kurdyka} and \cite[Example 5.3.17]{milzarek2016numerical}. (d) follows from the sum rule for polyhedral functions in \cite[Theorem 23.8]{rockafellar1970convex}. (e) follows from the representation of proper polyhedral functions as the sum of an indicator of a polyhedral set and the maximum of finitely many affine functions \cite[Proposition 3.1]{mordukhovich2016generalized}. {\cref{polyh} follows from \cite[Theorem 23.9]{rockafellar1970convex} if $\cR(A)\cap \dom g\neq \emptyset$, and if $\cR(A)\cap \dom g= \emptyset$, then both sides of the equality are empty. \cref{polyi} follows from \cite[Theorem 19.3]{rockafellar1970convex}.}
\end{proof}
\begin{prop}[{{Properties of Support Functions, \cite[Example 11.4(a)]{rockafellar2009variational}}}]
  \label{prop_support}
   Let $C\subseteq \R^n$ be a nonempty closed convex set. Then we have the following equivalences:
  \[  \bar v\in N_C(\bar x) \iff  \bar x \in \partial \sigma_C(\bar v) \iff  \sigma_C(\bar v)=\langle \bar v,\bar x\rangle,~\bar x \in C.       \]
\end{prop}
\begin{prop}
\label{prop2-2}
  Let $g:\R^n\to\R\cup\{\infty\}$ be a proper {polyhedral} convex function such that $\dom g\subseteq \R^n_+$. Assume $\bar{x}\in \dom g$, and let $I=\{i\in [n]:~\bar x_i\neq 0\}$. Then for all $v\in \partial g(\bar x)$ and any $w\in \R^n$ with $w_I=0$, $w_{I^c}\leq 0$, it holds that $v+w\in \partial g(\bar x)$. 
\end{prop}
\begin{proof}
  Since $\dom g \subseteq \R^n_+$, we have $g+\iota_{\R^n_+}=g$. By the sum rule of the subdifferential of polyhedral functions in \cref{polyf}, we have $\partial g=\partial g+\partial \iota_{\R^n_+}$. Since $\partial \iota_{\R^n_+}(\bar x)=N_{\R^n_+}(\bar x)=\{w\in \R^n:~w_I=0,~w_{I^c}\leq 0\}$, the conclusion follows.
\end{proof}
}

\section{First-order analysis}
\label{sec:foa}
{As can be seen from the definition of the KL property, establishing the subdifferential of $\Phi$ in \cref{sqprob1} is a key step to determine the KL exponent of $\Phi$. Although $\Phi$ is the composition of $\phi$ in \cref{prob1} and the square transformation, the classical chain rules {in \cref{fact_chainrule}} are not applicable since the square transformation does not admit a full-rank Jacobian in general. To handle this issue, we notice that for the KL property of $\Phi$, it is essential to characterize $\dist(0,\partial\Phi)$ rather than $\partial \Phi$. Here, we use variational techniques to show that $\dist(0,\partial\Phi)$ can be calculated using the naive chain rule. This relies on the subtle connections between the regular subdifferential and the subdifferential.} We start with the next preliminary result, which extracts the useful features of the subdifferential of a function that can be written as the composition of a polyhedral function and the square transformation. {Intuitively, we show that the naive chain rule $\partial (g\circ T)({x})\approx 2x\circ \partial g(x^2)$ at least holds at the components where $x_i\neq 0$.}
\begin{prop}
  \label{prop1}
  Let $h:\R^n\to\R\cup\{\infty\}$ be proper and lsc and set $H(x):=h(x^2)$ for all $x\in\R^n$. Assume $h(x^2)$ is finite and define $I:=\{i\in [n]:~x_i\neq 0\}$. Then, the following statements hold:
  \begin{itemize}
    \item[(i)]  For all $v\in \hat\partial H(x)$, we have $\bar v\in \hat\partial H(x)$, where $\bar v_i=v_i$ for $i\in I$, and $\bar v_i=0$ for $i\notin I$. 
    \item[(ii)] If in addition $h$ is polyhedral, then for all $v\in \partial H(x)$, there exists $w\in  \partial h(x^2)$, such that for all $i\in I$ it holds that $v_i=2x_iw_i$. 
  \end{itemize}
\end{prop}
\begin{proof}
   To simplify the notation, we assume $I=[r]$, and rewrite $x=(x_1;x_2)$ with $x_1\in \R^r$ and $x_2\in \R^{n-r}$. {Consider the matrix $Q\in \R^{n\times n}$ defined as $Q=\diag(I_r,-I_{n-r})$. Notice that $H(z)=H(Qz)$ for all $z\in \R^n$ and $Q$ is an invertible linear operator with $Q^{\top}=Q$. Since $x_2=0$, we know that $Qx=x$. We then can use the chain rule in {\cref{fact_chainrule}} to deduce that
   \[  \hat\partial H(x) =\hat\partial (H\circ Q)(x)=Q^\top \hat\partial H(Qx)=Q\hat\partial H(x).      \]
   This means that $Qv\in \hat \partial H(x)$ for all $v\in \hat\partial H(x)$, and hence $\bar v=(v+Qv)/2\in \hat \partial H(x)$ due to the convexity of $\hat \partial H(x)$ in \cref{fact_convex_hatpartial}. This proves (i). 
Next, we assume $h$ is polyhedral, and prove that 
\begin{equation}
  \label{eq:chain_regular_1}
   \forall v=(v_1;v_2)\in \hat\partial H(x),~~\exists w=(w_1;w_2)\in  \partial h(x^2),~~v_1=2x_1\circ w_1.  
\end{equation}
  Select arbitrary $v=(v_1;v_2)\in \hat\partial H(x)$. The definition of $\hat \partial H(x)$ gives that 
  \[    H(x_1',x_2')\geq H(x_1,0)+\langle v_1,x_1'-x_1  \rangle+\langle v_2,x_2' \rangle+o(\|x_1'-x_1\|+\|x_2'\|)                                \]
  Let $\bar H(x_1):=H(x_1,0)$, $\bar h(x_1):=h(x_1,0)$. Restricting $x_2'=0$ in the above inequality, we have $v_1\in \hat\partial \bar H(x_1)$. {Since the restriction of $T$ to the first $r$ coordinates is a diffeomorphism at $x_1$ because all components of $x_1$ are nonzero,} by the chain rule of the regular subdifferential in \cref{fact_chainrule} in the case of surjective Jacobian, we have $v_1\in \hat\partial \bar H(x_1)=2x_1\circ \hat\partial \bar h(x_1^2)=2x_1\circ \partial \bar h(x_1^2)$. Let $R\in \R^{n\times r}$ be defined as $R(x_1')=(x_1',0)$ for all $x_1'\in \R^r$, then $\bar h=h\circ R$. Since the chain rule holds for polyhedral function under linear mapping without any regularity condition by {\cref{polyh}}, we have 
  \[   \partial \bar h(x_1^2)=R^\top \partial h(x^2),\text{ and hence } \forall w_1\in \partial \bar h(x_1^2),~\exists (w_1,w_2)\in \partial h(x^2).     \]
  Thus there exists $w=(w_1;w_2)\in \partial h(x^2)$ such that $v_1=2x_1\circ w_1$, which proves \eqref{eq:chain_regular_1}. Essentially, the proof works for all $x\in \dom H$, and we can restate \eqref{eq:chain_regular_1} as
\begin{equation}
  \label{eq:chain_regular_2}
 \forall x'\in \dom H,~ \forall v\in \hat\partial H(x'),~~\exists w\in  \partial h((x')^2),~~ v_i=2x_i'w_i ~~ \forall i\in I_{x'},
\end{equation}
where $I_{x'}=\{i\in [n]:~x_i'\neq 0\}$. Next, select arbitrary $v=(v_1,v_2)\in \partial H(x)$. By the definition of $\partial H(x)$, we are able to find $x^k=(x^k_1;x^k_2)\to x$ and $v^k=(v^k_1;v^k_2)\to v$ with $v^k\in \hat\partial H(x^k)$. Since $x^k\to x$, we may assume {$x^k_1\neq 0$} for all $k\in\N$, and hence $[r]\subseteq I_{x^k}$ for all $k\in \N$. The result in \eqref{eq:chain_regular_2} implies that 
  \[   \exists w^k=(w_1^k;w_2^k)\in \partial h((x^k)^2),~~v_1^k=2x_1^k\circ w_1^k.                          \] 
  Using the inclusion property of the subdifferential of polyhedral functions \cref{polyd}, we may assume that
    \begin{equation}
      \label{eq10}
      \forall k\in \N,\quad \partial h((x^k)^2)\subseteq \partial h(x^2). 
    \end{equation}
    Therefore, we have $v_1^k/(2x_1^k)\in \Pi_r(\partial h((x^k)^2))\subseteq \Pi_r(\partial h(x^2))$,
    where $\Pi_r:\R^n\to \R^r$ is the natural projection onto the first $r$ components. 
    By {\cref{polya}}, $\partial h(x^2)$ is polyhedral. Thus $\Pi_r(\partial h(x^2))$ is also polyhedral by {\cref{polyi}}, and hence closed. Consequently, we have 
    \[ v_1/(2x_1)=\lim_{k\to\infty}v_1^k/(2x_1^k)\in \Pi_r(\partial h(x^2)).      \]
    This proves (ii).}
\end{proof}
Now, we are able to provide the characterization of $\dist(0,\partial\Phi)$.
\begin{prop}
  \label{prop1-2}
 Let {$\phi$ be defined in \cref{prob1} and $\Phi$ be defined in \eqref{sqprob1}}. Then for all $y\in \dom \partial\Phi$, it holds that 
 \[    \dist(0,\partial \Phi(y))=2\,\dist(-y\circ \nabla f(y^2),y\circ {\partial g(y^2)})=2\,\dist(0,y\circ \partial\phi(y^2)).            \]
\end{prop}
\begin{proof}
  By the sum rule in \cref{fact_sumrule}, we see that
  \begin{equation}
    \label{eq2-0}
   \forall y\in \dom\Phi,~ \partial \Phi(y)=2y\circ \nabla f(y^2)+\partial (g\circ T)(y),
  \end{equation}
  {Let $y\in \dom\partial\Phi$ be given. Since the subdifferential is always closed, there exists} a vector $v\in \partial (g\circ T)(y)$ such that $\dist(0,\partial\Phi(y))=\|v+2y\circ \nabla f(y^2)\|.$ Moreover, {using \cref{prop1}}, we know there exists $w\in \partial g(y^2)$ such that 
  \begin{equation}
    \label{eq2-1}
  \forall i\in I,~~2y_iw_i=v_i,\quad I=\{i\in [n]:~y_i\neq 0\}.
  \end{equation}
   In this case, we have:
  \begin{equation}
   \label{eq2-2} 
   \begin{aligned}
    \|v+2y\circ \nabla f(y^2)\|^2&=\sum_{i=1}^n(v_i+2y_i(\nabla f(y^2))_i)^2
    \\
    &\overset{\rm (a)}{=}\sum_{i\in I}(v_i+2y_i(\nabla f(y^2))_i)^2   +\sum_{i\notin I}v_i^2 \\
    &\geq \sum_{i\in I}(v_i+2y_i(\nabla f(y^2))_i)^2
    \\&\overset{\rm (b)}{=}\sum_{i\in I}(2y_iw_i+2y_i(\nabla f(y^2))_i)^2\\
    &\overset{\rm (c)}{=}\sum_{i=1}^n(2y_iw_i+2y_i(\nabla f(y^2))_i)^2\\
    &\overset{\rm (d)}{\geq} 4\dist^2(-y\circ \nabla f(y^2),y\circ \partial {g}(y^2)),
   \end{aligned}
  \end{equation}
  where in (a) and (c) we have used the fact that $y_i=0$ for all $i\notin I$, (b) follows from \eqref{eq2-1}, and (d) follows from $w\in \partial g(y^2)$. On the other hand, we also have
  \begin{align*}
    2y\circ \hat\partial {g}(y^2)\overset{\rm (a)}{=} 2y\circ\partial {g}(y^2) \overset{\rm (b)}{\subseteq} \hat\partial (g\circ T)(y) \subseteq  \partial (g\circ T)(y) ,
  \end{align*}
  {where (a) follows from the convexity of $g$ and (b) follows from the chain rule in \cref{fact_chainrule}.} This further implies that 
  \begin{align*}
    \dist(0,\partial\Phi(y))=\dist(-2y\circ \nabla f(y^2),\partial (g\circ T)(y))\leq 2\dist(-y\circ \nabla f(y^2),y\circ \partial {g}(y^2)).    
  \end{align*}
  Combined with \eqref{eq2-2}, we have finished the proof. The last equality $2\,\dist(-y\circ \nabla f(y^2),y\circ {\partial g(y^2)})=2\,\dist(0,y\circ \partial\phi(y^2))$ follows from the sum rule in \cref{fact_sumrule}.
\end{proof}
\section{Second-order variational analysis}
\label{sec:soa}
{Let $H=g\circ T$, where $g:\R^n\to \R\cup\{\infty\}$ is a polyhedral function with $\dom g\subseteq \R^n_+$. In general, calculating the second subderivative of $H$ would be difficult since the mapping $T$ does not have a full-rank Jacobian matrix. In this section, we aim to provide a characterization of the second subderivative of $H$ in a linear subspace.} We start with {calculating} the second subderivative of $H$ in a linear subspace.
\begin{prop}
  \label{prop:secsub}
  Let $g:\R^n\to \R\cup\{\infty\}$ be a proper polyhedral function and $H(y):=g(y^2)$ for all {$y\in\R^n$}. Let $\bar y\in \dom H$ and set $\bar x=\bar y^2$. Suppose $v\in \partial g(\bar x)$ and $\lambda =2 \bar y \circ v$. Define $I=\{i\in [n]:~\bar x_i\neq 0\}$ and $S_I:=\{z\in \R^n:~z_I=0\}$. Then,  
  \begin{equation}
    \label{eq_defn}
  \forall w\in S_I,\quad  d^2 H(\bar y|\lambda)(w)=2\sup_{p\in \cS(I,v)\cap \partial g(\bar x)}\langle  w_{I^c}^2,p_{I^c}\rangle ,
  \end{equation} 
  where {for any $J\subseteq [n]$ and $y\in \R^n$, we define} $\cS(J,y):=\{z\in \R^n:~z_J=y_J\}$.
\end{prop}
\begin{proof}
{Let 
$$Q:=\cS(I,v)\cap \partial g(\bar x)=\{p\in \partial g(\bar x):~2\bar y \circ p =\lambda\}.$$ Since $H=g\circ T$, by the lower bound of second subderivative of composite functions \cite[Corollary 4.3]{benko2023second}, we have 
  \begin{equation}
    \label{ineq6-1-5}
    \begin{aligned}
      d^2H(\bar y|2\bar y\circ v)(w)&\geq \sup_{p\in Q}2\langle  w^2,p\rangle+d^2g(\bar x|p)(2\bar y \circ w) \\
      &\overset{\rm (a)}{\geq} \sup_{p\in Q}2\langle  w^2,p\rangle  \overset{\rm (b)}{=} \sup_{p\in Q}2\langle  w_{I^c}^2,p_{I^c}\rangle 
    \end{aligned}
  \end{equation}
  where in (a) we have used {\cref{fact_convex_second_subderivative}} to show the second subderivative of a convex function is nonnegative, and in (b) we have used the fact that $w\in {S_I}$. Next, we prove the converse inequality. If $ \sup_{p\in Q}2\langle  w^2,p\rangle=\infty$, then we are done, since the {left-hand side} would also be $\infty$. Otherwise, we assume $\sup_{p\in Q}2\langle  w^2,p\rangle<\infty$, which means that $w^2\in \dom \sigma_Q$. Since $\sigma_Q$ is a proper polyhedral function due to $v\in Q$, using \cref{polya}, we know $w^2\in \dom \partial \sigma_Q$, and hence we can select $h\in \partial \sigma_Q(w^2)$. Using the properties of support functions in \cref{prop_support}, we have 
  $$ \langle w^2,h\rangle = \sigma_Q(w^2),~h\in Q,~w^2\in N_Q(h). $$ 
  Since $Q=\cS(I,v)\cap \partial g(\bar x)$ is the intersection of two polyhedral sets, by {the polyhedral sum rule in \cref{polyf} applied to the indicator functions of the two sets}, we have 
  \[    w^2\in N_Q(h)=N_{\cS(I,v)}(h)+N_{\partial g(\bar x)}(h)=\cS(I^c,0)+N_{\partial g(\bar x)}(h). \]
  In other words, there exists $q\in N_{\partial g(\bar x)}(h)$ such $q_{I^c}=w^2_{I^c}$. For {sufficiently small} $t>0$, we select $u(t)\in \R^n$ such that 
  \begin{equation}
    \label{eq6-1-0}
     2 \bar y\circ u(t)+t u(t)^2=tq.
  \end{equation}
  This gives an explicit construction of $u(t)$ as follows:
  \begin{equation}
  \label{eq6-1-1}
    \forall i\in I,~ u_i(t)=(-\bar y_i+{\sgn(\bar y_i)}\sqrt{\bar x_i+t^2q_i})/t ;\quad  \forall i \in I^c,~ u_i(t)=w_i. 
  \end{equation}
  Since $\bar x_i\neq 0$ for $i\in I$, $u(t)$ is well-defined for sufficiently small $t>0$. Moreover, we have $u(t)\to w$ as $t\downarrow 0$. Notice that $h\in Q$ implies $\bar y \circ v = \bar y \circ h$. Using the definition of second subderivative in \cref{def_second_subderivative}, we have 
  \begin{align*}
    d^2H(\bar y|2\bar y\circ v)(w)&=d^2H(\bar y|2\bar y\circ h)(w)\\
    &\leq \liminf_{t\downarrow 0}\frac{H(\bar y+tu(t))-H(\bar y)-t\langle 2\bar y\circ h,u(t)\rangle}{\frac{1}{2}t^2} \\ 
    & = \liminf_{t\downarrow 0}\frac{g((\bar y+tu(t))^2)-g(\bar x)-t\langle 2\bar y\circ h,u(t)\rangle}{\frac{1}{2}t^2} \\ 
    & \overset{\rm (a)}{=} \liminf_{t\downarrow 0}\frac{\sigma_{\partial g(\bar x)}(2t\bar y\circ u(t)+t^2u(t)^2)-2t\langle h, \bar y \circ u(t)\rangle }{\frac{1}{2}t^2}\\
    &\overset{\rm (b)}{=} \liminf_{t\downarrow 0}\frac{\sigma_{\partial g(\bar x)}(t^2q)-2t\langle h, \bar y \circ u(t)\rangle}{\frac{1}{2}t^2}\\
    &\overset{\rm (c)}{=} \liminf_{t\downarrow 0}\frac{t^2\langle  q,h\rangle -2t\langle h, \bar y \circ u(t)\rangle}{\frac{1}{2}t^2}\\
    &\overset{\rm (d)}{=} \liminf_{t\downarrow 0}\frac{ t\langle 2\bar y\circ u(t)+tu(t)^2,h \rangle-2t\langle h, \bar y \circ u(t)\rangle }{\frac{1}{2}t^2}\\
    &= 2\langle w^2,h\rangle \leq 2\sigma_Q(w^2). 
  \end{align*}
  where in (a) we have used the local support function representation of polyhedral functions in \cref{polye}, (b) follows from the choice of $u(t)$ in \cref{eq6-1-0}, (c) follows from $q\in N_{\partial g(\bar x)}(h)$ and the properties of support functions in \cref{prop_support}, and in (d) we recall the choice of $u(t)$ in \cref{eq6-1-0} again. Combined with \cref{ineq6-1-5}, the conclusion follows.
  }
\end{proof}
By using the characterization of the second subderivative of $\Phi$ in \cref{sqprob1}, we show in the next proposition how the second subderivative can be used to characterize the stationary points of $\Phi$ which correspond to stationary points of $\phi$ in \cref{prob1}. {Let us note that a similar result concerning the Hadamard parameterization for the $\ell_1$-norm regularized problem appears in \cite[Proposition 3.1]{ouyang2024kurdyka}.}
\begin{prop}
\label{prop4-3}
 Let $\phi$ be defined in \cref{prob1} and $\Phi$ be defined in \cref{sqprob1}. For all $y\in \R^n$ and $I:=\{i\in [n]:~y_i\neq 0\}$, the following are equivalent:
  \begin{itemize}
    \item[(i)] $y$ is a stationary point of $\Phi$ and for all $w\in S_I$ it holds that $d^2\Phi(y{|0})(w)\geq 0$, where $S_{I}:=\{v\in \R^n:~v_I=0\}$.
    \item[(ii)] $x:=y^2$ is a stationary point of $\phi$. 
  \end{itemize}
\end{prop}
\begin{proof}
  {Using {\cref{prop_sum_second_subderivatives}} and setting $G(t):=g(t^2)$ for all $t\in \R^n$, we see that for all $w\in \R^n$ it holds that}
  \begin{equation}
    \label{secsub_phi}
    d^2\Phi(y|0)(w)=d^2 G(y|-2y\circ \nabla f(x))(w)+4\langle y\circ w,\nabla^2f(x) y\circ w\rangle+2\langle \nabla f(x),  w^2 \rangle.
  \end{equation}
  Since for all $w\in S_I$ we have $y\circ w=0$, we further see that 
  \begin{equation}
    \label{secsub_phi1}
    \forall w\in S_I\quad  d^2\Phi(y|0)(w)=d^2 G(y|-2y\circ \nabla f(x))(w)+2\langle \nabla f(x),  w^2 \rangle  ,
  \end{equation}
  Suppose there exists $v\in \partial g(x)$ such that $-2y\circ \nabla f(x)=2 y \circ v $. By \cref{prop:secsub}, we have 
  \begin{equation}
    \label{eq6-3-2}
   \forall w\in S_I,\quad d^2 G(y| 2y\circ v)(w)= \sup_{p\in \cS(I,v)\cap \partial g({x})}2\langle p_{I^c},w_{I^c}^2 \rangle,
  \end{equation}
  where $\cS(I,v):=\{z\in \R^n:~z_I=v_I\}$. Combining \cref{secsub_phi1} and \cref{eq6-3-2} gives that  
  \begin{equation}
    \label{eq6-3-3}
   \forall w\in S_I,\quad d^2 \Phi(y| 0)(w)= \sup_{p\in \cS(I,v)\cap \partial g(x)}2\langle p_{I^c}+[\nabla f(x)]_{I^c},w_{I^c}^2 \rangle,
  \end{equation}
  where we have used the fact that $w_I=0$ since $w\in S_I$. 
  
  Assume (i). Since $y$ is a stationary point {of} $\Phi$, we see that $0\in \partial\Phi(y)$. Using \cref{prop1-2}, we see that there exists $v\in \partial g(x)$ such that $-2y\circ \nabla f(x)=2 y \circ v $. Therefore, \cref{eq6-3-3} holds. The condition $-2y\circ \nabla f(x)=2y\circ v$ is equivalent to that 
  \begin{equation}
    \label{fveq}
        -[\nabla f(x)]_{I}=v_I.
  \end{equation}
  Therefore, $\cS(I,v)=\cS(I,-\nabla f(x))$ and 
  $$(\cS(I,v)\cap \partial g( x))+\nabla f(x)=(\cS(I,v)+\nabla f(x))\cap (\partial g(x)+\nabla f(x)) =\cS(I,0)\cap \partial \phi(x), $$
  where we have used {the sum rule in} \cref{fact_sumrule} to show that $\partial \phi(x)=\nabla f(x)+\partial g(x)$. Consequently, \cref{eq6-3-3} yields that 
  \begin{equation}
    \label{eq6-3-4}
   \forall w\in S_I,\quad d^2 \Phi(y{|0})(w)= \sup_{p\in \cS(I,0)\cap \partial\phi({x})}2\langle  p_{I^c},w_{I^c}^2 \rangle,
  \end{equation}
  The condition $d^2\Phi(y{|0})(w)\geq 0$ for all {$w\in S_{I}$} means that 
  \[  \forall w\in S_{I}, \sup_{p\in \cS(I,0)\cap \partial\phi({x})}\langle  p_{I^c},w_{I^c}^2 \rangle\geq 0,\]
  or equivalently, 
  \begin{equation}
    \label{defnsetc}
    \forall q\in \R_+^{|I^c|},~~\sigma_{C}(q)\geq 0,\text{ where }C=\{p_{I^c}\in \R^{|I^c|}:~p\in \cS(I,0)\cap \partial \phi({x})\}. 
  \end{equation}
 We notice that the set $C$ is polyhedral since $\cS(I,0)$ and $\partial\phi(x)=\nabla f(x)+\partial g(x)$ are both polyhedral sets, and the projection of a polyhedral set is again polyhedral {(see \cref{polyi})}. Moreover, $C$ is nonempty since ${(\nabla f(x)+v)_{I^c}}\in C$ by \cref{fveq}, which further implies that $\sigma_C$ is proper polyhedral. This means that 
 \begin{align*}
  &0\in \argmin_{q\in \R^{|I^c|}} \sigma_C(q)+\iota_{\R^{|I^c|}_+}(q)  \iff 0\in \partial {(\sigma_C+\iota_{\R^{|I^c|}_+})(0)}\\
  &\overset{\rm (a)}{\iff} 0\in \partial \sigma_C(0)+\partial {\iota_{\R^{|I^c|}_+}}(0) \overset{\rm (b)}{\iff}    C\cap \R^{|I^c|}_+\neq \emptyset,
 \end{align*}
 where in (a) we have used that $0\in \dom \sigma_C\cap \dom{\iota_{\R^{|I^c|}_+}}$, and they are proper polyhedral to apply {\cref{polyf}}, in (b) we have used {\cref{prop_support}} to show $\partial \sigma_C(0)=C$ {and used} $\partial \iota_{\R^{|I^c|}_+}(0)=\R^{|I^c|}_-$. Consequently, by the definition of $C$ in \cref{defnsetc}, we see that 
 \begin{equation}
     \exists \lambda \in \partial\phi(x),~\lambda_{I^c}\geq 0,~ \lambda_I=0.
 \end{equation} 
 Then, by \cref{prop2-2}, we know $0\in \partial \phi({x})$, which proves (ii).

 Next, we assume (ii). By {the sum rule in \cref{fact_sumrule}}, we have 
 \[   \partial\phi(x)=\nabla f(x)+\partial g(x)\overset{\rm (a)}{=}\nabla f(x)+\hat\partial g(x)=\hat\partial \phi(x),               \]
 where (a) follows from {\cref{fact_convex_subdiff}}. Using {\cref{fact_chainrule}}, we see that 
 \[    0\in \partial \phi (x)=\hat\partial \phi(x),\quad  0=DT(y)^\top0 \subseteq DT(y)^\top\hat\partial\phi(x)   \subseteq \hat\partial\Phi(y)\subseteq \partial \Phi(y),   \]
 {This proves that $y$ is a stationary point of $\Phi$.} Moreover, let $w\in S_I$, using \cref{eq6-3-3}, we see that 
 \[d^2\Phi(y|0)(w)=\sup_{p\in \cS(I,-\nabla f(x))\cap \partial g(x)}2\langle p_{I^c}+[\nabla f(x)]_{I^c},w_{I^c}^2 \rangle \overset{\rm (a)}{\geq} 0,\]
 where in (a) we have used that $-\nabla f(x)\in \partial g(x)$ to show that 
 \[  -\nabla f(x)\in \cS(I,-\nabla f(x))\cap \partial g(x). \]
\end{proof}
\begin{exmp}
\label{exmp4-4}
    Consider \cref{nonnegative_cons_prob} and the corresponding reparameterized model \cref{un_sq_prob} \cite{ding2023squared}. Let $\bar x=\bar y^2$ and set $I=\{i\in [n]:~\bar y_i\neq 0\}$. By 
    \cref{prop4-3}, we see that $\bar x$ is a stationary point of $\phi$ in \cref{nonnegative_cons_prob} if and only if {$\bar y$ is a stationary point of $\Phi$ and}
    \begin{align*}
        d^2\Phi(\bar y{|0})(w)&\overset{\rm (a)}{=}\langle w,\nabla^2\Phi(\bar y)w\rangle\\
        &=2\langle \nabla f(\bar x), w^2\rangle+4\langle {\bar y}\circ w,\nabla^2f(\bar x)(\bar y\circ w) \rangle\\
        &\overset{\rm (b)}{=}2\langle \nabla f(\bar x),w^2\rangle \geq 0
    \end{align*}
     for $w\in S_I$, where in (a) we have used \cite[Example 13.8]{rockafellar2009variational} to show that $d^2\Phi(\bar y|0)(w)=\langle w, \nabla^2\Phi(\bar y) w\rangle$ for all $w\in \R^n$, and in (b) we have used the fact that for all $w\in S_I$ it holds that $\bar y\circ w=0$. Consequently, {every} second-order stationary point $y$ of $\Phi$ satisfies that $y^2$ is a stationary point of $\phi$. 
\end{exmp}
\begin{exmp}
    Consider \cref{simplex_cons_prob} and the corresponding reparameterized model \cref{ball_cons_prob} \cite{li2023simplex}. Let $\bar x,\bar y,$ and $I$ be defined as in \cref{exmp4-4}. By \cref{prop4-3}, $\bar x$ is a stationary point of $\phi$ if and only if $\bar y$ is a stationary point of $\Phi$ and $d^2\Phi(\bar y|0)(w)\geq 0$ for all $w\in S_I$. Since the unit sphere is $C^2$-cone reducible \cite[{Definition 6.1}]{mohammadi2021parabolic}, using {\cref{prop_sum_second_subderivatives}} and \cite[Theorem 6.2]{mohammadi2021parabolic}, for all $w\in \R^n$ we have 
    \begin{align*}
         d^2\Phi(\bar y|0)(w)\!&=\!4\langle \bar y \circ w,\nabla^2 f(\bar x) (\bar y\circ w)\rangle+2\langle \nabla f(\bar x),w^2\rangle+d^2\iota_{S^{n-1}}(\bar y|-2\bar y\circ \nabla f(\bar x))(w)\\
         &= 4\langle \bar y \circ w,\nabla^2 f(\bar x) (\bar y\circ w)\rangle+2\langle \nabla f(\bar x),w^2\rangle+\iota_{T_{\cM}(\bar y)}(w)+{\lambda\|w\|^2},
    \end{align*}
    where $\cM$ is the unit sphere in $\R^n$ and $\lambda\in \R$ satisfies that $-2\bar y\circ \nabla f(\bar x)=\lambda \bar y$. By \cite[p. 572]{davis2022proximal}, we know that for all $w\in T_{S^{n-1}}(\bar y)$, it holds that $\nabla^2 \Phi_{\cM}(\bar y)(w)=d^2\Phi(\bar y|0)(w)$, where {$\nabla^2 \Phi_{\cM}$ is the Riemannian Hessian of $\Phi$}. Therefore, if $\bar y$ is a second-order stationary point of $\Phi$ in the Riemannian setting \cite[Definition 3]{li2023simplex}, then $\bar x$ is a first-order stationary point of $\phi$ by \cref{prop4-3}. 
\end{exmp}
\begin{exmp}
    Consider \cref{prob1} with $g(x)=\iota_{S\cap \R^n_+}$ {and} $
    S=\{x\in \R^n:~Ax=b\}$ and $A\in \R^{m\times n}$, and its corresponding \cref{sqprob1} \cite{tang2024optimization}. Let $\bar y$ be a stationary point of \cref{sqprob1}, $\bar x=\bar y^2$ and $I=\{i\in[n]:~\bar y_i\neq 0\}$. {Set $G=g\circ T$.} Then we know $-2\bar y \circ \nabla f(\bar x)\in \partial G(\bar y)$, and by using \cref{prop1}, we know there exists $\lambda \in \partial g(\bar x)$ such that 
    \begin{equation}
        \label{grad_lambda}
        2\bar y\circ \lambda =-2\bar y\circ \nabla f(\bar x).
    \end{equation}
    By invoking {\cref{polyf} for the indicator functions of $S$ and $\R^n_+$}, there exist $\mu\in \R^{m}$ and $\rho\in \R^n_-$ with $\rho_{I}=0$ such that 
    \begin{equation}
        \label{decom_lambda}
        \lambda =A^\top \mu+\rho.
    \end{equation}
     According to the definition of $I$, it holds that $2\bar y\circ \lambda=2\bar y \circ (A^\top\mu)$. Set $\cP=\{y\in \R^n:~A(y^2)=b\}.$ Then $G=\iota_{\cP}$. Now, we are ready to calculate the second subderivative of $G$. Clearly, for all $w\notin T_{\cP}(\bar y)$, it holds that $d^2G(\bar y|2\bar y \circ \lambda)(w)=\infty$. For all $w\in T_{\cP}(\bar y)$, we have 
     \begin{equation}
         \label{exmp_sec_sub}
         \begin{aligned}
        d^2G(\bar y|2\bar y\circ \lambda)(w)&=d^2G(\bar y|2\bar y\circ (A^\top\mu))(w)\\
        &=\liminf_{\tilde w\to w,t_k\downarrow 0}\frac{G(\bar y+t_k\tilde w)-G(\bar y)-t_k\langle 2\bar y\circ (A^\top\mu),\tilde w \rangle}{\frac{1}{2}t^2_k} \\
        &=\liminf_{\overset{t_k\downarrow 0,\tilde w\to w}{\bar y+t_k\tilde w\in \cP}}\frac{-\langle  A^\top \mu ,2t_k\bar y\circ \tilde w \rangle}{\frac{1}{2}t_k^2} \\
        &=\liminf_{\overset{t_k\downarrow 0,\tilde w\to w}{\bar y+t_k\tilde w\in \cP}}\frac{-\langle  \mu ,A(2t_k\bar y\circ \tilde w )\rangle}{\frac{1}{2}t_k^2} \\
        &=\liminf_{\overset{t_k\downarrow 0,\tilde w\to w}{\bar y+t_k\tilde w\in \cP}}\frac{-\langle  \mu ,A((\bar y+t_k\tilde w)^2-\bar y^2 )-t_k^2A(\tilde w^2)\rangle}{\frac{1}{2}t_k^2} \\
        &=2\langle A^\top \mu, w^2 \rangle,
    \end{aligned}
     \end{equation}
    where in the last equation we have used that $A((\bar y+t_k\tilde w)^2)=A(\bar y^2)=b$. The previous calculation also shows that the second subderivative is independent of the choice of $\mu$ and $\rho$. Let $h\in T^2_{\cP}(\bar y, w)$, where $T^2_{\cP}$ is the second-order tangent set \cite[Definition 3.28]{bonnans2013perturbation}. Let $R$ be the range of the matrix $A\diag(\bar y)$, or equivalently $R$ is the linear subspace spanned by the column vectors of $A$ whose index is in $I$. \cite[Proposition 3]{tang2024optimization} shows that $\Pi_{R}A(\bar y\circ h+w^2)=0$. Notice that $A(\bar y \circ h)\in R$ and \cite[Proposition 2]{tang2024optimization} implies that $A(w^2)\in R$, we have 
    \begin{equation}
        \label{eqhw}
        A(\bar y\circ h) =-A(w^2).
    \end{equation}
    Let $F=f\circ T$. Consequently, we have 
    \begin{align*}
        \langle \nabla F(\bar y), h\rangle+\langle w,\nabla^2F(\bar y) w\rangle &=\langle 2\bar y\circ \nabla f(\bar x), h \rangle+\langle w,\nabla^2F(\bar y) w\rangle \\
        &=2\langle \nabla f(\bar x), \bar y\circ h \rangle+\langle w,\nabla^2F(\bar y) w\rangle \\
        &\overset{\rm (a)}{=}-2\langle A^\top\mu+\rho, \bar y\circ h \rangle+\langle w,\nabla^2F(\bar y) w\rangle \\
        &\overset{\rm (b)}{=}-2 \langle A^\top\mu , \bar y\circ h \rangle+\langle w,\nabla^2F(\bar y) w\rangle \\
        &=-2 \langle \mu , A(\bar y\circ h) \rangle+\langle w,\nabla^2F(\bar y) w\rangle \\
        &\overset{\rm (c)}{=}2 \langle \mu , A(w^2) \rangle+\langle w,\nabla^2F(\bar y) w\rangle  \\
        &\overset{\rm (d)}{=} d^2G(\bar y|-2\bar y\circ \nabla f(\bar x))(w)+\langle w,\nabla^2F(\bar y)w\rangle  \\
        &\overset{\rm (e)}{=} d^2\Phi(\bar y|0)(w).
     \end{align*}
    where in (a) we have used \cref{grad_lambda} and \cref{decom_lambda}, in (b) we have used the fact that $\rho_I=0$ and $\bar y_{I^c}=0$, (c) follows from \cref{eqhw}, (d) follows from \cref{exmp_sec_sub} and (e) follows from {\cref{prop_sum_second_subderivatives}}. Consequently, the second-order necessary condition for $\Phi$ defined in \cite[Definition 4]{tang2024optimization} is essentially the standard second-order necessary condition $d^2\Phi(\bar y|0)\geq 0$. Therefore, if $\bar y$ is a second-order stationary point in the sense of \cite[Definition 4]{tang2024optimization}, then we know $\bar x$ is a stationary point of $\phi$ by \cref{prop4-3}, which recovers \cite[Proposition 4]{tang2024optimization}.
\end{exmp}
{Except for the three examples above, in general, nonsmooth first-order optimizer often avoid strict saddle point in the sense of \cite[Definition 2.7]{davis2022proximal}, as illustrated in \cite{davis2022escaping}. In particular, if $\Phi$ has a $C^2$ active manifold $\cM$ at $\bar y$ in the sense of \cite[Definition 2.6]{davis2022proximal}, then nonsmooth first-order optimizer will avoid $\bar y$ if $\nabla^2 \Phi_{\cM}(\bar y)$ has a negative eigenvalue. Moreover, under mild assumptions, it is shown in \cite[Corollary 3.16]{hang2024fresh} that the second subderivative of $\Phi$ agrees with the quadratic form of Riemannian Hessian of $\Phi$ on the manifold $\cM$. Therefore, if $\nabla^2 \Phi_{\cM}(\bar y)$ has no negative eigenvalue, then the second subderivative of $\Phi$ is nonnegative. Consequently, if nonsmooth first-order optimizer converges to $\bar y$ where $\Phi$ has a $C^2$ active manifold, then the second subderivative of $\Phi$ at $\bar y$ is nonnegative under mild assumptions, which implies that $\bar x=\bar y^2$ is a stationary point of $\phi$ by \cref{prop4-3}.
}
\section{KL under strict complementarity condition}
\label{sec:KLstrict}
In this section, we deduce the KL exponent of $\Phi$ from that of $\phi$ {under the strict complementarity condition, i.e., $0\in \ri(\partial\phi(\bar x))$}. The key tool is the local representation of polyhedral functions \cref{polye}, i.e., {$g(y)-g(\bar x)=\sigma_{\partial g(\bar x)}(y-\bar x)$ for any $y$ in a neighborhood of $\bar x$}. Naturally, we will study the variational properties of the support function $\sigma_{\partial g(\bar x)}$ in order to understand the variational properties of $g$. {The next lemma provides the main tool to establish the KL exponent of $\Phi$ in \cref{sqprob1}. The main idea is to directly relate $\dist(0,\partial \Phi(y))$ to $\dist(0,\partial\phi(y^2))$ when $y^2-\bar x\in \parr(\partial g(\bar x))^\perp$, and move $y^2$ to the affine space $\{\bar x\}+\parr(\partial g(\bar x))^\perp$ when $y^2-\bar x\notin \parr(\partial g(\bar x))^\perp$.}
{
\begin{lemma}
  \label{nlemma5-1}
   Let $\phi$ and $\Phi$ be defined in \cref{prob1} and \cref{sqprob1}, respectively. Assume that $\bar x=\bar y^2$ and $0\in \ri(\partial\phi(\bar x))$. Let $I:=\{i\in [n]:~\bar x_i\neq 0\}$ and $\cL=\parr(\partial g(\bar x))^\perp$. Then, there exists a constant $\rho>0$ and a neighborhood $V$ of $\bar y$ such that for all $y\in V\cap \dom \partial \Phi$ it holds that 
   \begin{align*}
    \begin{cases}
      \dist(0,\partial\Phi(y))\geq (\min_{i\in I}|\bar y_i|)\dist(0,\partial\phi(y^2)) &\text{if } y^2-\bar x\in \cL,~{ I\neq\emptyset,}  \\
      \dist(0,\partial\Phi(y))\geq \rho\left(\dist(0,\partial \phi(\tilde x))+\|y^2-\tilde x\|^{\frac12}\right), &\text{if } y^2-\bar x\notin \cL,~{\text{or } I=\emptyset.}\\
    \end{cases}
   \end{align*}
   where $\tilde x:=\Pi_{\cL}(y^2-\bar x)+\bar x$.
\end{lemma}
\begin{proof}
  Our first step is to reduce the problem to a simple form. Define $C:=\partial g(\bar x)$. Replacing $f$ by $f-\langle \nabla f(\bar x),\cdot-\bar x \rangle$ and $g$ by $g+\langle \nabla f(\bar x),\cdot-\bar x \rangle$ if necessary, we may assume $\nabla f(\bar x)=0$ and $0\in \ri(\partial g(\bar x))$. By a change of variable, we can also assume $\bar y \geq 0$. Using the local representation of polyhedral functions in \cref{polye}, without loss of generality, $g(\cdot)=\sigma_C(\cdot-\bar x)$. In this case, $$0\in \ri(\partial g(\bar x))=\ri (\partial\sigma_C(0))=\ri(C),$$ which implies that $\parr(C)=\aff(C-0)=\spa(C).$ Also, using \cref{prop1-2}, we have $\dist(0,\partial\Phi(y))=2\dist(-y\circ \nabla f(y^2),y\circ \partial \sigma_C(y^2-\bar x))$ for all $y\in \dom \partial \Phi$. Using {the sum rule in} \cref{fact_sumrule}, we have $\dist(0,\partial \phi(x))=\dist(-\nabla f(x),\partial g(x))$ for all $x\in \dom \phi$.
  
  {In addition, we also consider the case where $I=\emptyset$, i.e., $\bar y=\bar x=0$. By the property of $\partial \sigma_C$ in \cref{prop2-2}, $C+\R_-^n\subseteq C$, so $\aff(C)=\R^n$, $\cL=\{0\}$, and $\tilde x=0$. Since $0\in\ri(C)=\operatorname{int}(C)$, we may choose $\epsilon>0$ such that $\epsilon\mathbf{1}\in C$. Since $\nabla f(0)=0$, we can also choose a neighborhood $V$ of $0$ such that $\|\nabla f(y^2)\|_\infty\leq\epsilon/2$ for all $y\in V$. Then, for any $y\in V\cap\dom\partial\Phi$, and any $v\in\partial\sigma_C(y^2)$, \cref{prop_support} and $\epsilon\mathbf{1}\in C$ give
  \[
    \begin{aligned}
      \langle y\circ(\nabla f(y^2)+v),y\rangle
      &=\langle \nabla f(y^2),y^2\rangle+\sigma_C(y^2)\geq(\epsilon-\|\nabla f(y^2)\|_\infty)\|y\|^2
      \geq\frac{\epsilon}{2}\|y\|^2.
    \end{aligned}
  \]
  The Cauchy--Schwarz inequality then yields $2\|y\circ(\nabla f(y^2)+v)\|\geq\epsilon\|y\|$. Taking the infimum over $v\in\partial\sigma_C(y^2)$ and using the bound of $\dist(0,\partial\Phi(y))$ in \cref{prop1-2}, we obtain
  \[
    \dist(0,\partial\Phi(y))\geq\epsilon\|y\|
    \overset{\rm (a)}{\geq\epsilon}\|y^2\|^{\frac12}
    \overset{\rm (b)}{=}\epsilon\left(\dist(0,\partial\phi(\tilde x))+\|y^2-\tilde x\|^{\frac12}\right),
  \]
  where (a) follows from$\|y^2\|\leq\|y\|^2$ and (b) follows from $0\in\partial\phi(\tilde x)=\partial\phi(0)$. This concludes the case where $I=\emptyset$. Therefore, in the following, we assume $I\neq\emptyset$.}

  Second, we provide a lower bound of $\dist(0,\partial\Phi(y))$. Let $v\in \partial\sigma_C(y^2-\bar x)$ such that 
  \begin{equation}
    \label{eq:v_def_dist}
    \dist(0,\partial\Phi(y))=2\dist(-y\circ \nabla f(y^2),y\circ \partial \sigma_C(y^2-\bar x))=2\|y\circ (\nabla f(y^2)+v)\|.
  \end{equation}
  Take a neighborhood $V$ of $\bar y$ such that for all $y\in V$ it holds that $\min_{i\in I}y_i^2\geq 1/4\min_{i\in I}\bar y_i^2$, and {$\|y^2-\bar x\|\leq \min_{i\in I}\min\{\bar y_i^4,\bar y_i^2\}$}. In this case, we have 
  \begin{equation}
    \label{lower_bound_dist_1}
    \begin{aligned}
          \dist^2(0,\partial\Phi(y))&=4\|y_I\circ (\nabla f(y^2)+v)_I\|^2+4\|y_{I^c}\circ (\nabla f(y^2)+v)_{I^c}\|^2\\
          &\geq \min_{i\in I}\{\bar y_i^2\}\|(\nabla f(y^2))_I+v_I\|^2+4\|y_{I^c}\circ (\nabla f(y^2)+v)_{I^c}\|^2.
    \end{aligned}
  \end{equation}

  Third, we analyze the geometric property of $\cL$. Since $\dom g \subseteq \R^n_+$, using \cref{prop2-2}, $\aff(C)\supseteq \{w\in \R^n:~w_I=0\}=:S_I$. Consequently, $\cL=\aff(C)^\perp\subseteq S_I^\perp = S_{I^c}$. 
   
  Fourth, we study the case where $d:=y^2-\bar x\in \cL=\spa(C)^\perp$. For all $c\in C$, it holds that $\langle c, d \rangle = \sigma_C(d)=0$. In this case, we can apply \cref{prop_support} to show that $\partial\sigma_C(d)=C$. Since $0\in \ri(C)$ and $\nabla f(\bar x)=0$, for all $y$ near $\overline{y}$, we know $p:=\Pi_{C}(-\nabla f(y^2))\in \ri(C)$. By the projection property and the fact that $p\in \ri(C)$, we have $-\nabla f(y^2)-p\perp \aff(C-p)=\aff(C)$, which means that $-\nabla f(y^2)-p\in \cL\subseteq S_{I^c}$. Then, for all $y\in V\cap \dom \partial\Phi$ we have 
  \begin{align*}
    \dist^2(-\nabla f(y^2),\partial \sigma_C(y^2-\bar x))&\overset{\rm (a)}{=}\dist^2(-\nabla f(y^2),C)\overset{\rm (b)}{=}\|\nabla f(y^2)+ p\|^2 \\
    &\overset{\rm (c)}{=} \|(\nabla f(y^2))_I+p_I\|^2  \overset{\rm (d)}{\leq} \inf_{q\in C}\|(\nabla f(y^2))_I+q_I\|^2 \\
    &\leq \|(\nabla f(y^2))_I+ v_I\|^2  \overset{\rm (e)}{\leq} \frac{1}{\min_{i\in I}\bar y_i^2}\dist^2(0,\partial \Phi(y))
  \end{align*}
   where (a) follows from $\partial\sigma_C(y^2-\bar x)=C$ as $y^2-\bar x\in \cL$, (b) follows from the projection definition of $p$, and (c) follows from the fact that $-\nabla f(y^2)-p\in S_{I^c}$. To justify (d), we need to prove $p_I=\Pi_{\cP(C)}(-(\nabla f(y^2))_I)$, where $\cP:\R^n\to \R^{|I|}$ is defined as $\cP x=x_I$. Since the set $\cP(C)$ is convex, we only need to prove that $-(\nabla f(y^2))_I-p_I\in N_{\cP(C)}(p_I)$. Using the normal cone representation of $\cP(C)$ in \cite[Theorem 6.43]{rockafellar2009variational}, i.e., $N_{\cP(C)}(p)=(\cP^*)^{-1}N_C(p)$, we have 
   $$\cP^*(-(\nabla f(y^2))_I-p_I)=-\nabla f(y^2)-p\in N_{C}(p), $$
   where we have used the fact $-\nabla f(y^2)-p\in S_{I^c}$. This justifies (d). The final (e) follows from \cref{lower_bound_dist_1}.

   Finally, we consider the case where $d\notin \cL$. In this case, we set $d=d_1+d_2$ with $d_1\in \cL,~d_2\in \cL^\perp=\aff(C)$. Since $0\in \ri(C)$, we may assume that $\bB_{2\delta}(0)\cap \aff(C)\subseteq C$. Reducing the neighborhood $V$ if necessary, we assume that for all $y\in V$ it holds that $\Pi_{\aff(C)}(-\nabla f(y^2))=\Pi_C(-\nabla f(y^2)){\in}~ \bB_\delta(0)\cap C$. Let $w:=\Pi_C(-\nabla f(y^2))=\Pi_{\aff(C)}(-\nabla f(y^2))$. Let us first consider the inner product $\langle d_2,\nabla f(y^2)+v\rangle$:
   \begin{align*}
      \langle d_2,\nabla f(y^2)+w+v-w\rangle\overset{\rm (a)}{=}\langle d_2,v-w \rangle \overset{\rm (b)}{=}\sup_{p\in C}\langle d_2,p-w\rangle \overset{\rm (c)}{\geq} \delta \|d_2\|,
   \end{align*}
   where in (a) we have used the fact $w=\Pi_{\aff(C)}(-\nabla f(y^2))$ to show that $-\nabla f(y^2)-w\in N_{\aff(C)}(w)=\cL$ and then $d_2\in \cL^\perp$ implies that $\langle d_2, \nabla f(y^2)+w\rangle=0$, in (b) we have used the fact $v\in \partial \sigma_C(d)=\partial\sigma_C(d_2)$ due to $d_2=d-d_1$ with $d_1\in C^\perp$, and in (c) we select $p=w+\delta d_2/\|d_2\|{\in\bB_{2\delta}\cap\aff(C)\subseteq C}$, {since $w,d_2\in\aff(C)$ and $\|p\|\leq\|w\|+\delta<2\delta$}. On the other hand, by Cauchy-Schwarz inequality, we have 
   \begin{align*}
     &|\langle d_2,\nabla f(y^2)+v\rangle|^2\leq \left( \sum_{i=1}^n|d_2|_i(\nabla f(y^2)_i+v_i)^2\right) \sum_{i=1}^n|d_2|_i\\
      &\overset{\rm (a)}{\leq}n^{\frac12}\left( \sum_{i\in I}|d_2|_i(\nabla f(y^2)_i+v_i)^2+ \sum_{i\in I^c}|d_2|_i(\nabla f(y^2)_i+v_i)^2 \right) \|d_2\|  \\
&\overset{\rm (b)}{\leq}n^{\frac12}\left(\|d_2\|_\infty\|(\nabla f(y^2))_I+v_I\|^2+\|y_{I^c}\circ((\nabla f(y^2))_{I^c}+v_{I^c})\|^2\right)\|d_2\| \\
&\overset{\rm (c)}{\leq}n^{\frac12}\left((\min_{i\in I}\bar y_i^2)\|(\nabla f(y^2))_I+v_I\|^2+\|y_{I^c}\circ((\nabla f(y^2))_{I^c}+v_{I^c})\|^2\right)\|d_2\|\\
&\overset{\rm (d)}{\leq}n^{\frac12}\|d_2\| \dist^2(0,\partial \Phi(y)),
\end{align*}
 where (a) follows from the power mean inequality, in (b) we have used the fact that $d_1,\bar x\in S_{I^c}$ to show that $y^2_{I^c}=(y^2-\bar x)_{I^c}=d_{I^c}=(d_2)_{I^c}$, (c) follows from the fact that $\|d_2\|_\infty\leq \|d_2\|\leq \|d\|\leq \min_{i\in I}\bar y_i^2$ by the choice of $V$, and (d) follows from \cref{lower_bound_dist_1}. Combining the above two inequalities, we have
\begin{equation}
  \label{lower_bound_dist_2}
  \dist(0,\partial \Phi(y))\geq n^{-\frac14}\delta\|d_2\|^{\frac12}.
\end{equation}
Notice that $y^2-\tilde x=y^2-\bar x- d_1=d_2$, from \cref{lower_bound_dist_1} we have:
\begin{align*}
  \dist(0,\partial \Phi(y))&\geq \min_{i\in I}\bar y_i\|(\nabla f(y^2))_I+v_I\|\\
  &\geq \min_{i\in I}\bar y_i(\|(\nabla f(\tilde x))_I+v_I\|-\|\nabla f(y^2)-\nabla f(\tilde x)\|) \\
  &\overset{\rm (a)}{\geq}  \min_{i\in I}\bar y_i(\|(\nabla f(\tilde x))_I+v_I\|-L\|y^2-\tilde x\|)\\ 
  &=  \min_{i\in I}\bar y_i(\|(\nabla f(\tilde x))_I+v_I\|-L\|d_2\|)   ,
\end{align*}
where in (a) we have used the locally Lipschitz continuity of $\nabla f$ due to $f\in C^2(\R^n)$. Moreover, since $\tilde x -\bar x\in \cL\subseteq S_{I^c}$, we know that $(\tilde x-\bar x)_{I^c}=0$, and hence $\tilde x_I>0$ and $\tilde x_{I^c}=0$ for sufficiently small $V$. Reducing $V$ if necessary, $\tilde y:=\sqrt{\tilde x}$ is well-defined. Next, we try to give a lower bound of $\|(\nabla f(\tilde x))_I+v_I\|$:
\begin{align*}
  \|(\nabla f(\tilde x))_I+v_I\|&\geq \frac{1}{\max_{i\in I}\tilde y_i}\|\tilde y_I\circ ( (\nabla f(\tilde x))_I+v_I)\|\\
  &\overset{\rm (a)}{=}  \frac{1}{\max_{i\in I}\tilde y_i}\|\tilde y\circ ( \nabla f(\tilde x)+v)\| \\
  &\overset{\rm (b)}{\geq}  \frac{1}{{4}\max_{i\in I}\bar y_i}\dist(-2\tilde y\circ \nabla f(\tilde x),2\tilde y\circ C) \\
  &~\overset{\rm (c)}{{\geq}}  \frac{\min_{i\in I}\bar y_i}{{4}\max_{i\in I}\bar y_i}\dist(0,\partial \phi(\tilde x)),
\end{align*}
where in (a) we have used the fact that $\tilde x-\bar x\in \cL\subseteq S_{I^c}$ and $\bar x_{I^c}=0$ to show that $\tilde x_{I^c}=0$, in (b) we choose $V$ to be sufficiently small such that ${\max_{i\in I}|\tilde y_i|}\leq 2 \max_{i\in I}|\bar y_i|$ and we note that $v\in C$, and in (c) we have used the conclusion for the case where $\tilde x-\bar x\in \cL$. Combining the above two inequalities with \cref{lower_bound_dist_2}, we have \begin{align*}
  2\dist(0,\partial \Phi(y))&\geq n^{-\frac14}\delta\|d_2\|^{\frac12}+ \frac{\min_{i\in I}\bar y_i^2}{{4}\max_{i\in I}\bar y_i} \dist(0,\partial \phi(\tilde x)) -L\min_{i\in I}\bar y_i\|d_2\| \\
  &\geq  \frac12 n^{-\frac14}\delta\|d_2\|^{\frac12}+ \frac{\min_{i\in I}\bar y_i^2}{{4}\max_{i\in I}\bar y_i} \dist(0,\partial \phi(\tilde x)),
  \end{align*}
  where in the last inequality we take a sufficiently small neighborhood $V$ such that $L\min_{i\in I}\bar y_i\|d_2\|^{\frac12}\leq 1/2n^{-\frac14}\delta$, which completes the proof.
\end{proof}
{\cref{nlemma5-1} connects $\dist(0,\partial\Phi(y))$ to $\dist(0,\partial \phi(\tilde x))$. The remaining task is to control the change of function values when the point moves from $y^2$ to $\tilde x$. This is done by utilizing the local Lipschitz continuity of smooth functions and the global Lipschitz continuity of polyhedral {functions on their domains \cref{polyg}}.}

\begin{thm}
\label{thmstrict}
 Let $\phi$ and $\Phi$ be defined in \cref{prob1} and \cref{sqprob1}, respectively. Suppose $0\in \ri(\partial \phi(\bar x))$ and assume $\bar x= \bar y^2$. Suppose that $\phi$ satisfies the KL property at $\bar x$ with exponent $\alpha\in [0,1)$. Then $\Phi$ satisfies the KL property at $\bar y$ with exponent $\max\{\alpha,1/2\}$.
\end{thm}
{
\begin{proof}
 As argued in the proof of \cref{nlemma5-1}, without loss of generality, $\nabla f(\bar x)=0\in \ri(C)$, where $C:=\partial g(\bar x)$ is a polyhedral set, $g(\cdot)=\sigma_C(\cdot-\bar x)$, and $\cL:=\parr(C)^\perp$. Moreover, we assume $\bar y\geq 0$. Let $I=\{i\in [n]:~\bar y_i> 0\}$. Take a neighborhood $V$ of $\bar y$ {and a constant $\rho_1>0$} such that the conclusion in \cref{nlemma5-1} holds. By \cref{defkl}, there exist constants $r,\ell,\rho>0$ such that for all $x\in \bB_r(\bar x)$ with $\ell>\phi(x)-\phi(\bar x)>0$, it holds that 
\begin{equation}
  \label{kl_cond}
  \dist(0,\partial \phi(x))\geq \rho(\phi(x)-\phi(\bar x))^\alpha. 
\end{equation}
Reduce $V$ such that $y^2\in \bB_r(\bar x)$ for all $y\in V$. Take arbitrary $y\in V$ such that $\Phi(y)-\Phi(\bar y)=\phi(y^2)-\phi(\bar x)\in (0,\ell/2)$. If $y^2-\bar x\in \cL$, then the result follows from \cref{nlemma5-1}. Otherwise, $y^2-\bar x\notin \cL$, by \cref{nlemma5-1} we have
\begin{equation*}
\dist(0,\partial \Phi(y))\geq \rho_1(\dist(0,\partial \phi(\tilde x))+\|y^2-\tilde x\|^{\frac12}),
\end{equation*}
where $\tilde x:=\bar x + \Pi_{\cL}(y^2-\bar x)$. Clearly, $\tilde x\in \dom g=\dom \phi$ since $g(\tilde x)=\sigma_C(\tilde x-\bar x)=0$. Moreover, $\|\tilde x-\bar x\|=\|\Pi_{\cL}(y^2-\bar x)\|\leq \|y^2-\bar x\|<r$ since $y\in V$, and hence $\tilde x\in \bB_r(\bar x)$. Since polyhedral functions are globally Lipschitz on their domains \cref{polyg} and $C^2$ functions are locally Lipschitz, we may assume that $\phi$ is Lipschitz continuous on $\bB_r(\bar x)\cap\dom \phi$ with constant $L>0$. Then we have 
\begin{equation}
  \label{lip_phi_br}
|\phi(y^2)-\phi(\tilde x)|\leq L\|y^2-\tilde x\|. 
\end{equation}
We may also reduce $r$ such that $L\|y^2-\tilde x\|<\ell/2$. Then, we have $\phi(\tilde x)\in (\phi(y^2)-\ell/2,\phi(y^2)+\ell/2)\subseteq (\phi(\bar x)-\ell/2,\phi(\bar x)+\ell)$ since $\phi(y^2)=\Phi(y)\in (\phi(\bar x),\phi(\bar x)+\ell/2)$ by the choice of $y$. If $\phi(\tilde x)\leq \phi(\bar x)$, then we have 
\[  \dist(0,\partial\Phi(y))\geq \rho_1\|y^2-\tilde{x}\|^{\frac12}\overset{\rm (a)}{\geq} \frac{\rho_1}{\sqrt{L}} |\phi(y^2)-\phi(\bar x)|^{\frac12}\overset{\rm (b)}{\geq} \frac{\rho_1}{\sqrt{L}}(\Phi(y)-\Phi(\bar y))^{\frac12},\]
where in (a) we have used \cref{lip_phi_br}, in (b) we have used the assumption that $\phi(y^2)\geq \phi(\bar x)\geq \phi(\tilde x)$, which yields $\Phi(y)-\Phi(\bar y)=\phi(y^2)-\phi(\bar x)\leq \phi(y^2)-\phi(\tilde x)$. 
Otherwise, $\phi(\tilde x)\in (\phi(\bar x),\phi(\bar x)+\ell)$. Combining \cref{nlemma5-1} and \cref{kl_cond}, we have
\begin{equation}
  \label{thmstrict_local_lb}
  \begin{aligned}
     \dist(0,\partial \Phi(y))
   &\geq \rho_1\left(\rho(\phi(\tilde x)-\phi(\bar x))^\alpha+\|y^2-\tilde x\|^{\frac12}\right)   \\
   &\overset{\rm (a)}{\geq} \rho_1\rho(\phi(\tilde x)-\phi(\bar x))^\alpha+ \frac{\rho_1}{\sqrt{L}}|\phi(y^2)-\phi(\tilde x)|^{\frac12} \\ 
   &\overset{\rm (b)}{\geq}  \rho_1\rho(\phi(\tilde x)-\phi(\bar x))^{\max\{\alpha,\frac12\}}+ \frac{\rho_1}{\sqrt{L}}|\phi(y^2)-\phi(\tilde x)|^{\max\{\alpha,\frac12\}} \\
   &\overset{\rm (c)}{\geq} \min\{\rho_1\rho,\frac{\rho_1}{\sqrt{L}}\}(\phi(y^2)-\phi(\bar x))^{\max\{\alpha,\frac12\}},
  \end{aligned}
\end{equation}
where (a) follows from the Lipschitz continuity of $\phi$ on $\bB_r(\bar x)\cap \dom \phi$ in \cref{lip_phi_br}, in (b) we reduce $r$ further such that for all $z_1,z_2\in \bB_r(\bar x)\cap \dom \phi$ it holds that $|\phi(z_1)-\phi(z_2)|\leq 1$, in (c) we have used the concavity of the function $a\mapsto a^{\max\{\alpha,1/2\}}$, and $h(a+b)\leq h(a)+h(b)$ for any concave function $h$ with $h(0)=0$ and for all $a,b\in \R_+$. This together with the case where $y^2-\bar x\in \cL$ finishes the proof.  
\end{proof}
}
\section{KL without strict complementarity condition}
\label{sec:KLnonstrict}
In this section, we derive the KL exponent of $\Phi$ in \cref{sqprob1} without assuming {the strict complementarity condition}. The techniques used in this section are similar to \cite{ouyang2024kurdyka}. {More precisely, the proof of the next lemma follows the same template as \cite[Lemma 3.10]{ouyang2024kurdyka}: one combines the KL inequality at $\bar x$ with the cross-section error bound involving the sets $S_\rho$ to derive a lower bound for the weighted quantity $\sum_{i\in I}|v_i|^2+\sum_{i\in I^c}|x_i-\bar x_i||v_i|^2$. This argument only needs a choice of $v\in\partial g(x)$ and estimates on its components, so differentiability of $g$ is not used. In addition, since there is no explicit separable block here, the notation in \cite[{Lemma 3.10}]{ouyang2024kurdyka} reduces to the case $J_1=\emptyset$.}
\begin{lemma} 
  \label{lemma5-1}
  Suppose that $g:\R^n\to\R\cup\{\infty\}$ is convex and $\bar x \in \R^n$ is a global minimizer of $g$. Fix an index set $I\subseteq [n]$. Let $\Omega:=\argmin g$. Assume that
  \begin{itemize}
      \item   {There exists a neighborhood $U$ of $\bar x$, a constant $c>0$, and $\gamma\in (0,1]$ such that for all $\rho\in \R^{n}_{+}$,} it holds that
  \begin{equation}\label{J3eb}
{\dist(x,\Omega\cap S_{\rho}) \leq c\max\{\dist(x,\Omega),\dist(x,S_{\rho})\}^\gamma}\ \ \ \forall x\in U\cap \dom\partial g,
  \end{equation}
  where {$S_{\rho}:=\{x\in\R^n:|x_i-\bar x_i|\leq \rho_i\ \ \ \forall i\in I^c\}$}.
  \end{itemize}
   Suppose that $g$ satisfies the KL property at $\bar x$ with exponent $\alpha\in(0,1)$. 
   Then, by shrinking $U$ if necessary, there is a constant $\sigma>0$ such that
  \[
   \sum_{i\in I}|v_i|^2+\sum_{i\in I^c}|x_i-\bar x_i||v_i|^2\geq \sigma(g(x)-g(\bar x))^{1+\beta}\ \ \ \forall x\in U\cap \dom \partial g,~v\in\partial g(x),
  \]
  where $\beta=1-\gamma(1-\alpha)\in (0,1)$.
\end{lemma}
\begin{proof}
  There is no loss of generality if we assume $\bar x=0$. Then $g(0)=\inf g$. In view of the KL assumption and \cite[Proposition 2.3]{ouyang2024kurdyka}, by shrinking the $U$ in \eqref{J3eb} further if necessary, we can deduce that there exists $c_1 > 0$ such that
  \begin{equation}\label{eqeb}
    g(x) - g(0)\geq c_1^{-\frac1{1-\alpha}}\dist^{\frac{1}{1-\alpha}}(x,\Omega)\ \ \ \forall x\in U\cap\dom\partial g.
  \end{equation}
  Shrink $U$ further if necessary such that $\|x\|\leq 1$ for all $x\in U$. The desired result is trivial if $g(x)=g(0)$, so we only need to consider the case where $g(x)>g(0)$.

  Fix any $x\in U\cap\dom\partial g$ with $g(x)>g(0)$ and
  \begin{equation}
      \label{condU}
     g(x)-g(0)\leq 1.
  \end{equation}
  Fix any $v\in\partial g(x)$. We define
  \begin{equation}\label{eqindex}
  \mathcal I:=\left\{i\in I^c:~|x_i|\leq c_2(g(x) - g(0))^{\gamma(1-\alpha)}\right\},
  \end{equation}
  where $c_2=cc_1^\gamma$, with $c_1$ defined in \eqref{eqeb} and $c$ and $\gamma$ defined in \eqref{J3eb}. We define $\rho\in \R_{+}^{n}$ as $\rho_i=|x_i|$ for all $i\in [n]$. Then, by \cref{J3eb} and notice that $x\in S_\rho$, we see that
  \begin{equation}\label{xomegaieb}
    \begin{aligned}
    \dist(x,\Omega\cap S_{\rho}) &\leq c\max\{\dist(x,S_{\rho}),\dist(x,\Omega)\}^\gamma=c\,\dist^\gamma(x,\Omega)\\
    &\overset{\rm (a)}{\leq} cc_1^\gamma(g(x)-g(0) )^{\gamma({1-\alpha})}=c_2 (g(x)-g(0) )^{\gamma(1-\alpha)},
    \end{aligned}
\end{equation}
where (a) follows from \eqref{eqeb}. Additionally, let $\hat x=\Pi_{\Omega\cap S_{\rho}}(x)$, and we know that $\|x-\hat x\|=\dist(x,\Omega\cap S_{\rho})$. We further define the following three index sets:
       \begin{equation}
        \label{defn_indexi3}
        I_1:=\{i\in\mathcal I:~\rho_i=0\},~I_2:=\mathcal I\setminus I_1,~I_3:=I^c\setminus\mathcal I.
       \end{equation}
      Then for $i\in I_2\cup I_3$ we have $|x_i|>0$ due to $\rho_i=|x_i|$ for $i\in I^c$ and \cref{eqindex}, and moreover,
       \begin{equation}
        \label{boundxi}
        \begin{aligned}
          &\forall i\in I_1,~ \hat x_i=\rho_i=0,~|x_i-\hat x_i|=|x_i|\overset{\rm{(a)}}{\leq} c_2(g(x) - g(0))^{\gamma(1-\alpha)},\\
          &\forall i\in I_2,~ \frac{1}{|x_i|} |x_i-\hat x_i|^2\overset{\rm{(b)}}{\leq} \frac{4}{|x_i|} |x_i|^2=4|x_i|\overset{\rm{(c)}}{\leq} 4c_2(g(x) - g(0))^{\gamma(1-\alpha)},  \\
          &\forall i\in I_3,~\frac{1}{|x_i|} |x_i-\hat x_i|^2\overset{\rm (d)}{\leq} \frac{c_2^2(g(x)-g(0))^{2\gamma(1-\alpha)}}{|x_i|}\overset{\rm (e)}{\leq} c_2(g(x) - g(0))^{\gamma(1-\alpha)},
        \end{aligned}
       \end{equation}
       where in (a) and (c) we have used the fact that $I_1,I_2\subseteq\mathcal I$ in \eqref{defn_indexi3} and the definition of $\mathcal I$ in \eqref{eqindex}, (b) follows from the facts that $\hat x\in S_{\rho}$ and that for all $y\in S_{\rho},~i\in I^c$ it holds that $|x_i-y_i|\leq |x_i|+|y_i|\leq 2|x_i|$, (d) holds because of the relation $|x_i-\hat x_i|\leq \|x-\hat x\|=\dist(x,\Omega\cap S_\rho)$ and the bound in \eqref{xomegaieb}, and (e) follows from the definition of $I_3$ in \eqref{defn_indexi3} and that of $\mathcal I$ in \eqref{eqindex}. Therefore, by using the convexity of $g$, we can deduce that:
           \begin{align*}
            &g(x)-g(0) = g(x) - g(\hat x) \\
             &\overset{\rm(a)}{\leq}  \sum_{i=1}^n|v_i||x_i-\hat x_i|\overset{\rm(b)}{=}\sum_{i\in I\cup I^c}|v_i||x_i-\hat x_i| \\
             &\overset{\rm(c)}{=} \sum_{i\in I}|v_i||x_i-\hat x_i|+\sum_{i\in I_1} |v_i||x_i|^{\frac12}|x_i|^{\frac12}\!\! +\!\!\!\!\sum_{i\in I_2\cup I_3}|v_i|\sqrt{|x_i|} \frac{1}{\sqrt{|x_i|}}|x_i-\hat x_i| \\
             &\overset{\rm(d)}{\leq} \sqrt{\sum_{i\in I}|v_i|^2+\sum_{i\in I^c}|x_i||v_i|^2 }\sqrt{\sum_{i\in I}|x_i-\hat x_i|^2+\sum_{i\in I_1}|x_i|+\sum_{i\in I_2\cup I_3}\frac{|x_i-\hat x_i|^2}{|x_i|}}  \\
             &\overset{\rm(e)}{\leq} c_4\sqrt{\sum_{i\in I}|v_i|^2+\sum_{i\in I^c}|x_i||v_i|^2 }\cdot (g(x)-g(0))^{\frac{\gamma(1-\alpha)}{2}}
            \end{align*}
  where (a) follows from the convexity of $g$, which implies that $g(x)-g(\hat x)\leq \langle v, x-\hat x\rangle$, (b) follows from $I\cup I^c=[n]$, (c) follows from the definition of the index sets $I_1,I_2$, and $I_3$ in \eqref{defn_indexi3}, (d) follows from the Cauchy-Schwarz inequality,  (e) follows from the facts that $\sum_{i\in I}|x_i-\hat x_i|^2\leq \|x-\hat x\|^2\leq c_2^2(g(x)-g(0))^{2\gamma(1-\alpha)}\leq c_2^2(g(x)-g(0))^{\gamma(1-\alpha)}$ (which follows from \eqref{xomegaieb} and \eqref{condU}), \eqref{boundxi}, and we set $c_4=2\sqrt{n}\max\{c_2,\sqrt{c_2}\}$. Rearrange this inequality and notice that $\beta=1-\gamma(1-\alpha)$, we obtain
  \begin{align*}
    \frac{1}{c_4^2}(g(x)-g(0))^{1+\beta} \leq   \sum_{i\in I}|v_i|^2+\sum_{i\in I^c}|x_i||v_i|^2 .
  \end{align*}
  For $x\in U\cap\dom\partial g$ with $g(x)-g(0)>1$ and $v\in\partial g(x)$, the convexity of $g$ and the Cauchy--Schwarz inequality give
  \begin{align*}
    g(x)-g(0)&\leq\langle v,x\rangle\leq\sqrt{\sum_{i\in I}|v_i|^2+\sum_{i\in I^c}|x_i||v_i|^2}\sqrt{\|x_I\|^2+\sum_{i\in I^c}|x_i|}\\
    &\leq\sqrt{1+\sqrt n}\sqrt{\sum_{i\in I}|v_i|^2+\sum_{i\in I^c}|x_i||v_i|^2}.
  \end{align*}
  Since $1+\beta<2$, this gives
  \[
    \frac{(g(x)-g(0))^{1+\beta}}{1+\sqrt n}\leq\sum_{i\in I}|v_i|^2+\sum_{i\in I^c}|x_i||v_i|^2.
  \]
 Therefore, the result follows with $\sigma=\min\{(4n\max\{c^2c_1^{2\gamma},cc_1^\gamma\})^{-1},(1+\sqrt n)^{-1}\}$.
\end{proof}
\begin{thm}
  \label{thm-KLnonstrict}
  Assume that $\phi$ in \cref{prob1} is convex. Let $\bar x=\bar y^2$ be a stationary point of $\phi$ and $\Omega:=\argmin_{x\in\R^n}\phi(x)$. Suppose $I:=\{i\in [n]:~\bar x_i\neq 0\}$.  Assume further there exists a neighborhood $U$ of $\bar x$ and constants $c>0$ and $\gamma\in (0,1]$  such that for all $\rho\in \R_+^n$ it holds that
  \begin{equation}
    \label{assumeb}
    \dist(x,\Omega\cap S_\rho)\leq c\max\{\dist(x,\Omega),\dist(x,S_\rho)\}^\gamma  \quad \forall x\in U\cap \dom\partial \phi,
  \end{equation}
  where $S_\rho=\{x\in \R^n:~|x_i-\bar x_i|\leq \rho_i\quad \forall i \in I^c\}$. If $\phi$ satisfies the KL property with exponent $\alpha\in (0,1)$ at $\bar x$, then $\Phi$ in \cref{sqprob1} satisfies the KL property with exponent $\frac{1+\beta}{2}$, where $\beta=1-\gamma(1-\alpha)$. 
\end{thm}
\begin{proof}
  By \cref{lemma5-1}, we know there exists a neighborhood $V$ of $\bar x$ such that 
 \begin{equation}
  \label{eq5-2-0}
  \sum_{i\in I}|v_i|^2+\sum_{i\in I^c}|x_i||v_i|^2\geq \sigma(\phi(x)-\phi(\bar x))^{1+\beta}\ \ \ \forall x\in V\cap \dom \partial \phi,~v\in\partial \phi(x),
 \end{equation}
 where {$\sigma>0$ and} we have used the definition of $I$ to show that $\bar x_i=0$ for all $i\in I^c$. Select a neighborhood $U$ of $\bar y$ such that 
  \begin{subequations}
    \begin{align}
      \forall y\in U,~i\in I,~~~ |y_i|\geq \delta{>0}, \label{eq5-2-2a} \\
      \forall y\in U,\quad y^2\in V, \label{eq5-2-2b}
    \end{align}
  \end{subequations}
  Fix $y\in U\cap \dom\partial \Phi$ and set $x=y^2$. Using \cref{prop1-2}, we know there exists $v\in \partial \phi(x)$ such that
  \begin{equation}
    \label{eq5-2-3}
    \begin{aligned}
      \dist^2(0,\partial \Phi(y))&=4\|y\circ v\|^2=4\sum_{i=1}^n y_i^2v_i^2\overset{\rm (a)}{=}4\sum_{i\in I} y_i^2v_i^2+4\sum_{i\in I^c} |x_i|v_i^2 \\
      &\overset{\rm (b)}{\geq} 4\sum_{i\in I}\delta^2v_i^2+4\sum_{i\in I^c} |x_i|v_i^2 \\
      &\geq 4\min\{\delta^2,1\}\left(\sum_{i\in I}v_i^2+\sum_{i\in I^c} |x_i|v_i^2\right) \\
      &\overset{\rm (c)}{\geq} 4\min\{\delta^2,1\}\sigma(\phi(x)-\phi(\bar x))^{1+\beta} \\
      &=4\min\{\delta^2,1\}\sigma(\Phi(y)-\Phi(\bar y))^{1+\beta}
    \end{aligned}
  \end{equation} 
  where in (a) we have used $x=y^2$, in (b) we have used \cref{eq5-2-2a}, and (c) follows from \cref{eq5-2-2b} and \cref{eq5-2-0}. This finishes the proof.
\end{proof}
}
\begin{remark}
\label{remark6-3}
  Let us note that the exponent given in \cref{thm-KLnonstrict} is tight {for $\gamma\in (0,1/2]$ and $\alpha\in (1/2,1)$}, as can be seen in \cite[{Example D.1}]{ouyang2024kurdyka}. In particular, this covers the case where $f=h(Ax)$ with $h\in C^2(\R^n)$ being strictly convex, in which case the solution set is a polyhedral set by \cite[Lemma 3.2]{ouyang2024kurdyka}\footnote{{The solution set is given by $\nabla f^{-1}(v)\cap \partial g^{-1}({-v})$ for some $v\in \R^n$, where $\partial g^{-1}({-v})$ is polyhedral since $g$ is polyhedral, and $\nabla f^{-1}(v)$ is {an affine space}.}}, and then the conclusion in \cref{thm-KLnonstrict} holds with $\gamma=1$ by \cite[Lemma 3.2.3]{facchinei2007finite}.
\end{remark}
\begin{remark}
  Suppose $f$ in \cref{prob1} is a composition of a {locally strongly} convex function, {i.e., strongly convex on any compact set,} and a linear mapping. Then $\phi$ in \cref{prob1} satisfies the KL property with exponent $1/2$ \cite[Corollary 5.1]{li2018calculus}.  This covers the cases where $f$ is the logistic loss function or the least-squares loss function. {In this case, by \cref{remark6-3}, $\gamma=1$ in \cref{J3eb}.} Therefore, by using \cref{thmstrict} and \cref{thm-KLnonstrict}, we know the KL exponent of $\Phi$ in \cref{sqprob1} at its {second-order} stationary point is either $1/2$ or $3/4$, depending on whether the strict complementarity condition holds or not.
\end{remark}
\section{\texorpdfstring{{Applications}}{Applications}}
\label{sec:app}
{
In this section, we demonstrate applications of our KL results to practical algorithms. We start with the connection to the square transformation, discuss {a hybrid first-order algorithm} for solving \cref{sqprob1}, and show how the KL exponent of $\Phi$ in \cref{sqprob1} can be used to establish its convergence rates. {The KL exponent of $\Phi$ can also be used to derive function-value convergence rates for the proximal gradient method \cite[Theorem 3.4]{frankel2015splitting} and a monotone accelerated proximal gradient method \cite[Theorem 2]{li2017apgnc} applied to \cref{sqprob1}, under their respective assumptions.}

\subsection{\texorpdfstring{{Connections with Hadamard Parametrization}}{Connections with Hadamard Parametrization}}
Consider the Hadamard parameterization, that is $x=u\circ v$. By using the Hadamard parameterization, the problem $\min_{x\in \R^n}f(x)$ is transformed to $\min_{u,v\in \R^n} F(u,v):=f(u\circ v)$. Using the idea of Hadamard difference parameterization \cite{vaskevicius2019implicit}, by doing the invertible linear transformation $a=(u+v)/2$ and $b=(u-v)/2$, we can equivalently transform the problem $\min_{u,v\in \R^n} f(u\circ v)$ to $\min_{a,b\in \R^n} G(a,b):=f(a^2-b^2)$. Since invertible linear transformation does not change the KL exponent, the KL exponent of $F$ is the same as the KL exponent of $G$ at the corresponding point. Moreover, $G$ can be viewed as the square transformation of the function $\widehat f(x_1,x_2):=f(x_1-x_2)+\iota_{\R^n_+\times \R^n_+}(x_1,x_2)$.  
\begin{prop}
  Let $f\in C^1(\R^n)$. Define $\widehat f:\R^n\times \R^n\to \R\cup\{\infty\}$ as 
  \[  \widehat f(x_1,x_2)=f(x_1-x_2)+\iota_{\R^{n}_+\times \R^n_+}(x_1,x_2).  \]
  Then, for any $x_1,x_2\in \R^n_+$ and $x=x_1-x_2$, we have 
  \[ \sqrt{2}\|\nabla f(x)\| \geq \dist(0,\partial \widehat f(x_1,x_2))\geq \|\nabla f(x)\|,~\widehat f(x_1,x_2)=f(x). \]
  In particular, for all $x\in \R^n$, the following statements are equivalent:
  \begin{enumerate}[label=\rm (\roman{*})]
    \item $f$ has KL exponent $\alpha\in [0,1)$ at $x$; 
    \item $\widehat f$ has KL exponent $\alpha$ at $(x_+, x_-)$ for any $(x_+,x_-)\in \R^{n}_+\times \R^n_+$ satisfying $x=x_+-x_-$. 
    \item $\widehat f$ has KL exponent $\alpha$ at $(x_+, x_-)$ for some $(x_+,x_-)\in \R^{n}_+\times \R^n_+$ satisfying $x=x_+-x_-$. 
  \end{enumerate}
  Moreover, if $\Omega:=\argmin_x f(x)$ is nonempty, then $\argmin_{x_1,x_2}\widehat f=Q^{-1}(\Omega)\cap (\R^n_+\times \R^n_+)$, where $Q(x_1,x_2)=x_1-x_2$.
\end{prop}
\begin{proof}
 We have
 \begin{align*}
  \dist^2(0,\partial\widehat f(x_1,x_2))&=\dist^2(-\nabla f(x),N_{\R^n_+}(x_1))+\dist^2(\nabla f(x),N_{\R^n_+}(x_2))\\
  &\overset{\rm (a)}{\geq}  \dist^2(-\nabla f(x),\R^n_-)+\dist^2(\nabla f(x),\R^n_-)=\|\nabla f(x)\|^2, 
 \end{align*}
 where (a) follows from the fact that $N_{\R^n_+}(z)\subseteq \R^n_-$ for all $z\in \R^n_+$. The upper bound $2\|\nabla f(x)\|^2$ can be derived by noticing that $\{0\}\subseteq N_{\R^n_+}(z)$ for all $z\in \R^n_+$. The rest of the statements directly follow from the bound on $\dist(0,\partial\widehat f(x_1,x_2))$ and the equality of the function values. The structure of the solution set is also trivial. 
\end{proof}
Consequently, we can now apply \cref{thm-KLnonstrict} to $\widehat f$ to derive the KL exponent of $G(a,b)=f(a^2-b^2)$ at $(\bar a,\bar b)$ for any $(\bar a,\bar b)\in \R^n\times \R^n$ with $\bar a^2-\bar b^2=\bar x$.
\begin{corollary}
   Assume that $f\in C^2(\R^n)$ is convex, and has KL exponent {$\alpha\in (0,1)$} at $\bar x\in \Omega:=\argmin_x f(x)$. Assume further that $\Omega$ is a nonempty polyhedral set. Then, $G(a,b):=f(a^2-b^2)$ has KL exponent $(1+\alpha)/2$ at any $(\bar a,\bar b)\in \R^n\times \R^n$ with $\bar a^2-\bar b^2=\bar x$. 
\end{corollary}
Next, let us consider the $\ell_1$-norm regularized problem $\min_{x\in \R^n} f(x)+\mu \|x\|_1$ with $\mu>0$, which is considered in \cite{ouyang2024kurdyka}. By using the Hadamard difference parameterization, we can equivalently transform this problem to $\min_{a,b\in \R^n} f(a^2-b^2)+\mu(\|a\|^2+\|b\|^2)$, which is the square transformation of the problem 
$$\min_{x_1,x_2\in \R^n}\widetilde f(x_1,x_2)=f(x_1-x_2)+\mu(\mathbf{1}^\top(x_1+x_2))+\iota_{\R^n_+\times \R^n_+}(x_1,x_2).$$ 
The next proposition connects the KL exponent of $\widetilde f$ and that of $f(\cdot)+\mu \|\cdot\|_1$.
\begin{prop}
  Let $\widetilde f:\R^n\times \R^n\to \R\cup\{\infty\}$ be defined as above. Then, for any $x_1,x_2\in \R^n_+$ and $x=x_1-x_2$, if $\min\{x_1,x_2\}=0$, then we have
  \[  \dist(0,\partial \widetilde f(x_1,x_2))\geq \dist(-\nabla f(x),\mu\partial\|x\|_1),~\widetilde f(x_1,x_2)={f(x)+\mu\|x\|_1}.      \]
  If $\min(x_1,x_2)\neq 0$, then we also have
  \[  \dist(0,\partial \widetilde f(x_1,x_2))\geq \sqrt{2}\mu.     \] 
In addition, the following statements hold.
\begin{enumerate}[label=\rm{(\roman{*})}]
  \item For all $x\in \R^n$, if $f(\cdot)+\mu \|\cdot\|_1$ has KL exponent $\alpha\in {(0,1)}$ at $x$, then $\widetilde f$ has KL exponent $\alpha$ at $(x_+, x_-)$ for any $(x_+,x_-)\in \R^{n}_+\times \R^n_+$ satisfying $x=x_+-x_-$.
  \item If $(x_1,x_2)$ is a stationary point of $\widetilde{f}$ and $-\nabla f(x)\in \mu\ri(\partial\|x\|_1)$, then $0\in \ri({\partial \widetilde f(x_1,x_2)})$. 
  \item If $\Omega:=\argmin_x f(x)+\mu \|x\|_1$ is nonempty, then
  {\[
\arg\min_{x_1,x_2}\widetilde f
=
\left\{
(x_1,x_2)\in\mathbb R_+^n\times\mathbb R_+^n:
x_1-x_2\in\Omega,\quad
x_1+x_2=|x_1-x_2|
\right\}.
\]}
\end{enumerate}
\end{prop}
\begin{proof}
Since 
\begin{equation}
  \label{sub_tildef}
  \partial \widetilde f(x_1,x_2)=\begin{bmatrix}
    \nabla f(x_1-x_2)+\mu \mathbf{1}+N_{\R^n_+}(x_1) \\ 
    -\nabla f(x_1-x_2)+\mu \mathbf{1}+N_{\R^n_+}(x_2)
  \end{bmatrix},
\end{equation}
we have
  \begin{align*}
   &\dist^2(0,\partial\widetilde f(x_1,x_2))=\dist^2(-\nabla f(x)-\mu \mathbf{1},N_{\R^n_+}(x_1))+\dist^2(\nabla f(x)-\mu\mathbf{1},N_{\R^n_+}(x_2))\\
    &=\sum_{i=1}^n \dist^2(-(\nabla f(x))_i-\mu ,N_{\R_+}((x_1)_i))+\dist^2((\nabla f(x))_i-\mu ,N_{\R_+}((x_2)_i))
  \end{align*}
 If there exists an index $i\in [n]$ such that $\min\{(x_1)_i,(x_2)_i\}>0$, then we have 
 \begin{align*}
   &\dist^2(0,\partial\widetilde f(x_1,x_2))\\
   &\geq  \dist^2(-(\nabla f(x))_i-\mu ,N_{\R_+}((x_1)_i))+\dist^2((\nabla f(x))_i-\mu ,N_{\R_+}((x_2)_i)) \\
   & = |-(\nabla f(x))_i-\mu |^2 +|(\nabla f(x))_i-\mu|^2\geq 2\mu^2.
 \end{align*}
  Now fix index $i\in [n]$. If $x_i>0$, then $(x_{1})_i>0$ since $(x_{1})_i=x_i+(x_{2})_i\geq x_i$, in which case we have 
 \begin{equation}\label{eq7-1}
  \begin{aligned}
    \dist^2(-(\nabla f(x))_i-\mu ,N_{\R_+}({(x_1)_i}))&=|-(\nabla f(x))_i-\mu|^2\\
    &=\dist^2(-(\nabla f(x))_i ,\mu\partial|x_i|).  
  \end{aligned}
 \end{equation}
 Next, we assume $\min\{x_1,x_2\}=0$. Fix $i\in [n]$. If $x_i<0$, then similarly we have $(x_{2})_i>0$, and hence 
   \begin{equation}\label{eq7-2}
   \dist^2((\nabla f(x))_i-\mu ,N_{\R_+}((x_2)_i))=|(\nabla f(x))_i-\mu|^2=\dist^2(-(\nabla f(x))_i ,\mu\partial|x_i|).  
 \end{equation}
  If $x_i=0$, then we have
  \begin{equation}\label{eq7-3}
    \begin{aligned}
       &\dist^2(-(\nabla f(x))_i-\mu ,N_{\R_+}((x_1)_i))+\dist^2((\nabla f(x))_i-\mu ,N_{\R_+}((x_2)_i)) \\
    &\geq \dist^2(-(\nabla f(x))_i-\mu ,\R_-)+\dist^2((\nabla f(x))_i-\mu ,\R_-)\\
    &=\dist^2(-\nabla f(x)_i ,[-\mu,\mu])=\dist^2(-\nabla f(x)_i,\mu\partial|0|).
    \end{aligned}
  \end{equation}
  Consequently, we have $\dist^2(0,\partial \widetilde f(x_1,x_2))\geq \dist^2(-\nabla f(x),\mu\partial\|x\|_1)$. (i) directly follows from this inequality and \cite[Remark 2.2]{ouyang2024kurdyka}. Now let us prove (ii). {Let $I=\{ i\in [n]:~x_i\neq 0\}$ and $(x_1,x_2)$ be a stationary point of $\widetilde{f}$. This gives $\min\{x_1,x_2\}=0$, because otherwise we have $\dist(0,\partial\widetilde f(x_1,x_2))\geq \sqrt{2}\mu>0$. The condition $-\nabla f(x)\in \mu\ri(\partial\|x\|_1)$ means that 
  \begin{equation}
    \label{eq7-5}
    \begin{cases}
      -\nabla f(x)_i= \mu x_i/|x_i|, & i\in I,\\
      -\nabla f(x)_i\in (-\mu,\mu), & i\notin I.
    \end{cases}
  \end{equation}
  Also, $x=x_1-x_2$, $x_1,x_2\in \R^n_+$ and $\min\{x_1,x_2\}=0$ imply that
  \begin{equation*}
    \begin{cases}
      (x_1)_i=x_i>0,~(x_2)_i=0, &\text{when } x_i>0,\\
      (x_1)_i=0,~(x_2)_i=-x_i>0, &\text{when } x_i<0,\\
      (x_1)_i=(x_2)_i=0, & i\notin I.
    \end{cases}
  \end{equation*}
  For $t>0$, $\ri(N_{\R_+}(t))=\{0\}$, while $\ri(N_{\R_+}(0))=(-\infty,0)$. Therefore, when $x_i>0$ we have $ -\nabla f(x)_i=\mu$ from \cref{eq7-5} and hence
  \[
  \begin{aligned}
    -\nabla f(x)_i-\mu=0\in\ri(N_{\R_+}((x_1)_i)),~\nabla f(x)_i-\mu=-2\mu\in\ri(N_{\R_+}((x_2)_i)).
  \end{aligned}
  \]
  When $x_i<0$ we have $-\nabla f(x)_i=-\mu$ from \cref{eq7-5} and hence
  \[
  \begin{aligned}
    -\nabla f(x)_i-\mu=-2\mu\in\ri(N_{\R_+}((x_1)_i)),~\nabla f(x)_i-\mu=0\in\ri(N_{\R_+}((x_2)_i)).
  \end{aligned}
  \]
  When $x_i=0$ we have $-\nabla f(x)_i\in(-\mu,\mu)$ from \cref{eq7-5} and hence
  \[
  \begin{aligned}
    -\nabla f(x)_i-\mu,~\nabla f(x)_i-\mu\in(-2\mu,0)\subseteq\ri(N_{\R_+}((x_1)_i))=\ri(N_{\R_+}((x_2)_i)).
  \end{aligned}
  \] 
  In summary, we have $0\in \ri(\partial \widetilde f(x_1,x_2))$. Finally,} (iii) can be proved by noticing that for $x_1-x_2=x$, $\widetilde f(x_1,x_2)\geq f(x)+\mu \|x\|_1$, and equality holds if and only if $x_1+x_2=|x|$.
\end{proof}
This means that we can deduce the KL exponent of $$G(a,b)=f(a^2-b^2)+\mu(\|a\|^2+\|b\|^2)$$ at $(\bar a,\bar b)$ for any second-order stationary point $(\bar a,\bar b)\in \R^n\times \R^n$ from the KL exponent of $f(\cdot)+\mu\|\cdot\|_1$ at $\bar a^2-\bar b^2=\bar x$. We can recover the result in \cite[Theorem 3.8]{ouyang2024kurdyka}. Moreover, we can also recover \cite[{Theorem 3.11}]{ouyang2024kurdyka} when $\Omega:=\argmin_x f(x)+\mu \|x\|_1$ is polyhedral. The general case where $\Omega$ is not polyhedral requires more dedicated analysis on the error bound condition in \cref{assumeb} since the solution set has been changed to the solution set of $\widetilde f$, which is beyond the scope of this paper.

\subsection{\texorpdfstring{{Convergence analysis of a hybrid first-order algorithm}}{Convergence analysis of a hybrid first-order algorithm}}
In this subsection, we consider the case where $g=\iota_{\cX}$ and $\cX:=\{x\in\R^n:~Ax=b,~x\geq 0\}$. This is considered in \cite{tang2024optimization}. We assume $A\in \R^{m\times n}$ has full row rank throughout this subsection. The authors in \cite{tang2024optimization} proposed a hybrid first-order method for solving \cref{sqprob1} in this case, which combines the projected gradient method and the Riemannian gradient method. To introduce their algorithm, we need to introduce the manifold language. Throughout this subsection, we work with smooth submanifolds of the ambient Euclidean space $\R^n$, equipped with the metric induced by the Euclidean inner product. If $\cM\subseteq \R^n$ is such a submanifold and $h:\R^n\to \R$ is continuously differentiable, then for each $y\in \cM$ the tangent space $T_y\cM$ is a linear subspace of $\R^n$, and the Riemannian gradient of $h$ on $\cM$ is simply the orthogonal projection of the Euclidean gradient onto $T_y\cM$, namely
\begin{equation}
  \label{def_rgrad}
  \grad_{\cM} h(y)=\Pi_{T_y\cM}\nabla h(y).
\end{equation}
Moreover, if $\cM$ admits a local representation
\[
  \cM=\{y\in \R^n:~c(y)=0\},
\]
where $c:\R^n\to \R^m$ is continuously differentiable and $Dc(y)$ has full row rank at any point in $\cM$, then
\[
  T_y\cM=\{d\in \R^n:~Dc(y)d=0\}.
\]
For the update step, we also need the notion of a retraction. In our Euclidean setting, a retraction at $y\in \cM$ is simply a smooth mapping $R_y:T_y\cM\to \cM$ such that $R_y(0)=y$ and $DR_y(0)=\mathrm{Id}$\footnote{The notation $\mathrm{Id}$ means the identity mapping.}. Thus, $R_y$ maps a tangent vector back to the manifold while preserving the first-order information at $y$. Accordingly, a Riemannian gradient step takes the form
\[
  y^+=R_y(-\alpha \grad_{\cM} h(y)),
\]
where $\alpha>0$ is the stepsize. Therefore, the Riemannian gradient method amounts to moving along the tangent space of the corresponding smooth stratum and then retracting the trial point back to that stratum. In the present setting, the feasible set of \cref{sqprob1} also admits a finite stratification\footnote{In the literature \cite[Definition 8.53]{ioffe2017variational}, $\{\cX_I\}_I$ is called a stratification of $C$, if $C$ is a disjoint union of $\{\cX_i\}_{i\in I}$, and the frontier condition holds: if $\overline{\cX_i}\cap \cX_j\neq \emptyset$ with $i\neq j$, then $\cX_j\subseteq \overline{\cX_i}$. In our setting this frontier condition also holds, but it is not used in later proofs, so we do not discuss it here.} into smooth submanifolds. For each $I\subseteq [n]$, let
\[
  \cX_I:=\{x\in \R^n:~Ax=b,~x_I>0,x_{I^c}=0\}.
\]
Clearly, $\cX_I$ is either empty or a smooth submanifold of $\R^n$. Moreover, $\{\cX_I\}_{I\subseteq [n]}$ gives a finite stratification of $\cX$ (possibly with some empty sets). Moreover, $T$ has a surjective Jacobian at any point in $\cY_I:=T^{-1}(\cX_I)$ by restricting the {domain and} codomain to $\{x\in\R^n:~x_{I^c}=0\}$. Therefore, $\cY_I$ is also a smooth submanifold of $\R^n$ by the regular level set theorem \cite[Corollary 5.14]{Lee00}. Indeed, $T^{-1}(\cX_I)$ has several connected components depending on the sign of $y$, each of which is also a smooth submanifold. In this case, by using the tangent cone representation in \cite[Exercise 6.7]{rockafellar2009variational}, we have 
\begin{equation}
  \label{yi_tn_cones}
  \begin{aligned}
     \forall y\in \cY_I,\quad &T_{\cY_I}(y)=\{d\in \R^n:~ A(y\circ d)=0,~d_{I^c}=0\},\\
     &N_{\cY_I}(y)=T_{\cY_I}(y)^\perp=\{y\circ (A^\top \lambda)+\mu\in \R^n:~\lambda\in \R^m,~\mu_{I}=0\}.
  \end{aligned}
\end{equation}
This special structure gives the next connection between the Riemannian gradient of $\Phi$ on $\cY_I$ and the minimal subdifferential of $\partial\Phi(y)$:
\begin{equation}
  \label{connect_riemannian_minnorm}
  \forall y\in \cY_I,~\grad_{\cY_I} \Phi(y)\in \argmin_{v\in \partial\Phi(y)}\|v\|.
\end{equation}
Indeed, by the definition of Riemannian gradient in \cref{def_rgrad}, we have 
$$\grad_{\cY_I}\Phi(y):=\Pi_{T_{\cY_I}(y)}(2y\circ \nabla f(y^2))=2y\circ \nabla f(y^2)+v,~v\in \argmin_{v\in N_{\cY_I}(y)}\|2y\circ \nabla f(y^2)+v\|.  $$
Given the representation of normal cone in \cref{yi_tn_cones} and notice that $y_{I^c}=0$ for all $y\in \cY_I$, we can further rewrite $\grad_{\cY_I}\Phi(y)$ as 
\begin{equation}
  \label{rg}
\grad_{\cY_I}\Phi(y)=2y\circ (\nabla f(y^2)+A^\top\lambda),~\lambda\in  \argmin_{\rho\in \R^m}\|2y\circ (\nabla f(y^2)+A^\top\rho)\|.  
\end{equation}
On the other hand, by the sum rule of subdifferentials of polyhedral functions \cref{polyf}, we have 
\begin{equation}
  \label{normal_cx}
  \begin{aligned}
  N_{\cX}(y^2)&=N_{\{x\in \R^n:~Ax=b\}}(y^2)+N_{\R^n_+}(y^2)\\
  &=\{A^\top\rho+\mu:~\rho\in \R^{m},~\mu_I=0,~\mu_{I^c}\leq 0\}.
\end{aligned}
\end{equation}

Since $y_{I^c}=0$, we know for all $A^\top\rho+\mu$ with $\mu_I=0$, it holds that $y\circ (A^\top\rho+\mu)=y\circ (A^\top\rho)$, and hence 
\[ A^\top\lambda\in   \argmin_{v\in \partial g(y^2)=N_\cX(y^2)}\|2y\circ (\nabla f(y^2)+v)\|.\] 
Applying \cref{prop1-2}, we see that \cref{connect_riemannian_minnorm} holds.

Next, let us introduce the hybrid first-order algorithms. Fix parameters $\eta,\beta\in (0,1)$ and initial trial steps $\alpha_0,t_0>0$. The two stepsizes are chosen by the Armijo line-search rules
\begin{subequations}
  \begin{align}
        \Phi\big( y(\alpha)\big)\leq \Phi(y) -\eta \|\grad_{\cM}\Phi(y)\|\|y(\alpha)-y\|,~y(\alpha)=R_y(-\alpha \!\grad_{\cM}\Phi(y)),
  \label{eq:hybrid_armijo_rgd} \\
    \phi\big(\Pi_{\cX}(x-t\nabla f(x))\big)
    \leq \phi(x)-\eta t^{-1}\|\Pi_{\cX}(x-t\nabla f(x))-x\|^2.
  \label{eq:hybrid_armijo_pgd}
  \end{align}
\end{subequations}
Let us note that \cref{eq:hybrid_armijo_pgd} follows \cite[Equation 38]{tang2024optimization}, while \cref{eq:hybrid_armijo_rgd} is different from \cite[Equation 33]{tang2024optimization} since we replace $\alpha\|\grad_{\cM}\Phi(y)\|^2$ by $ \|\grad_{\cM}\Phi(y)\|\|y(\alpha)-y\|$. This modification is for the global convergence proof in \cref{convergence_1}. Nevertheless, by the property of retraction $DR_y(0)=\mathrm{id}$, $\|y(\alpha)-y\|\approx  \alpha \|\grad_{\cM}\Phi(y)\|$ when $\alpha$ is sufficiently small, and hence the line-search procedure would also end after finitely many steps as proved in \cite[Proposition 10]{tang2024optimization}. We summarize the resulting algorithms in \cref{alg:hybrid_first_order}.
\begin{algorithm}[ht]
\caption{A simplified hybrid first-order method}
\label{alg:hybrid_first_order}
\begin{algorithmic}[1]
\Require $y^0\in T^{-1}(\cX)$ and $x^0=(y^0)^2$.
\For{$k=0,1,2,\dots$}
\If{$k\in J$}
\State choose the smallest $\ell_k\in \n$ such that, with $t_k=t_0\beta^{\ell_k}$,
\[
  x^{k+1}:=\Pi_{\cX}(x^k-t_k\nabla f(x^k))
\]
satisfies the Armijo condition \eqref{eq:hybrid_armijo_pgd} at $x=x^k$ and $t=t_k$;
\State choose $y^{k+1}=\sqrt{x^{k+1}}$.
\Else
\State set $I:=\{i\in [n]:~y_i^k\neq 0\}$. 
\State choose the smallest $\ell_k\in \n$ such that, with $\alpha_k=\alpha_0\beta^{\ell_k}$,
\[
  y^{k+1}:=R_{y^k}(-\alpha_k \grad_{\cY_I}\Phi(y^k))
\]
satisfies the Armijo condition \eqref{eq:hybrid_armijo_rgd} at $y=y^k$, $\cM=\cY_I$, and $\alpha=\alpha_k$;
\State set $x^{k+1}=(y^{k+1})^2$.
\EndIf
\EndFor
\end{algorithmic}
\end{algorithm}

The above algorithm keeps the main idea of \cite{tang2024optimization}: the cheaper Riemannian gradient step exploits the smooth geometry of the current stratum, while the more expensive proximal gradient step allows the iteration to move across different strata and directly enforces descent for the original problem in the variable $x$. Compared with the adaptive switching rule in \cite{tang2024optimization}, our presentation uses a prescribed infinite set $J$ as an abstraction of the periodic insertion strategy mentioned in \cite[Remark 8]{tang2024optimization}.
\begin{prop}
  \label{convergence_1}
 Consider the optimization problem \cref{sqprob1} with $g=\iota_{\cX}$. Let $\{y^{k}\}_k$ be generated by \cref{alg:hybrid_first_order}, and assume $f$ is subanalytic. If $\{y^k\}_k$ is bounded, then $\{x^k\}_k$ converges to a stationary point of $\phi$ in \cref{prob1}. 
\end{prop}
\begin{proof}
   Since this is a descent algorithm, we know $\{\Phi(y^k)\}_k$ is monotonically decreasing. Moreover, $\Phi$ is continuous on its domain, and $\{y^k\}_k$ is bounded, so $\{\Phi(y^k)\}_k$ converges to some $\Phi^*$. {If $\Phi(y^{k_0})=\Phi^*$ for some $k_0$, then $\Phi(y^k)=\Phi^*$ for all $k\ge k_0$. By \eqref{eq:hybrid_armijo_rgd}, every subsequent Riemannian step leaves $y^k$ unchanged. By \eqref{eq:hybrid_armijo_pgd}, every subsequent proximal gradient step leaves $x^k$ unchanged. Since $J$ is infinite, choose $j\in J$ with $j\ge k_0$. Then $\Pi_{\cX}(x^j-t_j\nabla f(x^j))-x^j=0$, so $-\nabla f(x^j)\in N_{\cX}(x^j)$ by \cite[Theorem 6.41]{beck2017first}, and hence $x^j$ is stationary for $\phi$. For all $k\ge j$, $x^k=x^j$, so $y^k$ is stationary for $\Phi$ by \cref{prop4-3}. Thus $y^k$ converges to a stationary point of $\Phi$, which is $\sqrt{x^j}$. Therefore, in the following we assume $\phi(x^k)=\Phi(y^k)>\Phi^*$ for all $k$ and $\Phi(y^k)\downarrow\Phi^*$.} Next, we prove that $\{x^{k+1}\}_{k\in J}$ has an accumulation point which is a stationary point of $\phi$. Let $x^*$ be an accumulation point of $\{x^{j+1}\}_{j\in J}$. Select a compact convex set $K$ such that $\{x^k\}_k\subseteq K$. Since $f\in C^2(\R^n)$, we may assume that $\nabla f$ is Lipschitz continuous on $K$. By a similar analysis as in \cite[Proposition A.1(ii)]{chen2016penalty}, we know $\{t_k\}_{k\in J}$ is lower bounded by some $\underline{t}>0$. By the Armijo condition \eqref{eq:hybrid_armijo_pgd}, we have 
   \[  \phi(x^k) -\phi(x^{k+1})\geq  \frac{\eta}{t_k}\|x^{k+1}-x^k\|^2\geq  \frac{\eta}{t_0} \|x^{k+1}-x^k\|^2.  \]
   Since $\{\phi(x^k)\}_k$ is monotonically decreasing and convergent, we have  
   \[  \sum_{k\in J}\|x^{k+1}-x^k\|^2\leq \frac{t_0}{\eta}(\phi(x^0)-\Phi^*)<\infty.    \]
   This implies that $\|x^{k+1}-x^k\|\to 0$ as $k\in J$ tends to infinity. Let $\{x^{j_k+1}\}_k$ be a subsequence of $\{x^{j+1}\}_{j\in J}$ converging to $x^*$. Then, we also have $x^{j_k}\to x^*$. Using the monotonicity of the residual of the proximal gradient mapping on the parameter $t$ \cite[Theorem 10.9]{beck2017first}, we have 
   \begin{align*}
    \|x^{j_k+1} -x^{j_k}\|&=\|\Pi_{\cX}(x^{j_k}-{t_{j_k}}\nabla f(x^{j_k}))-x^{j_k}\|\\
    &\geq \|\Pi_{\cX}(x^{j_k}-\underline{t}\nabla f(x^{j_k}))-x^{j_k}\|\to \|\Pi_{\cX}(x^{*}-\underline{t}\nabla f(x^{*}))-x^{*}\|=0,
   \end{align*}
   which implies that $x^*$ is a stationary point of $\phi$ by \cite[Corollary 10.8]{beck2017first}. Our next task is to prove that the whole sequence $\{x^k\}_k$ converges to $x^*$. Since $f$ is continuous and subanalytic, we know the graph of $f$ on $K$ is bounded subanalytic, and hence globally subanalytic. Therefore, $\phi$ and $\Phi$ are also globally subanalytic by restricting the domain on $K$ and $T^{-1}(K)$, respectively. Applying the KL condition \cite[Corollary 16]{bolte2007clarke}  of $\Phi$ at each point in $T^{-1}(K)\cap \{y\in \R^n:~\Phi(y)=\Phi^*\}$ (nonempty since $\sqrt{x^*}$ lies in this set), which is a closed set since $\dom \Phi=T^{-1}(\dom \phi)$ is closed and $\Phi$ is continuous on its domain, utilizing the compactness of $T^{-1}(K)\cap \{y\in \R^n:~\Phi(y)=\Phi^*\}$, we can find constants $\epsilon>0$, $c>0$ and $\alpha\in [0,1)$ such that for all $y\in T^{-1}(K)$ with $|\Phi(y)-\Phi^*|<\epsilon$ and $\dist(y,T^{-1}(K)\cap \{y\in \R^n:~\Phi(y)=\Phi^*\})<\epsilon$, it holds that
   \[  \dist(0,\partial\Phi(y)) \geq c(\Phi(y)-\Phi^*)^\alpha,~\dist(0,\partial\phi(y^2))\geq c(\phi(y^2)-\Phi^*)^\alpha,     \]
   where the second inequality follows from the KL condition of $\phi$ at each point in $K\cap \{x\in\R^n:~\phi(x)=\Phi^*\}$. Since $\Phi(y^k)\to \Phi^*$, by removing a finite number of terms, we may assume that for all $k$, $\Phi(y^k)-\Phi^*<\epsilon$. Also, argue by contradiction and utilizing the continuity of $\Phi$ on $\dom \Phi$, we may assume that $\dist(y^k,T^{-1}(K)\cap \{y\in \R^n:~\Phi(y)=\Phi^*\})<\epsilon$ for all $k\in \n$. Let $\tau^k:=(\Phi(y^k)-\Phi^*)^{1-\alpha}/(c(1-\alpha))$. By concavity of the function $t\mapsto t^{1-\alpha}$ on $\R_+$, we have
   \[ \forall k\in \n,\quad \tau^{k}-\tau^{k+1}\geq \frac{\Phi(y^k)-\Phi(y^{k+1})}{c(\Phi(y^k)-\Phi^*)^{\alpha}}\geq \begin{cases}
      \frac{\eta\|\grad_{\cY_I}\Phi(y^k)\|\|y^{k+1}-y^k\|}{\dist(0,\partial\Phi(y^k))},& k\notin J,\\
      \frac{\eta}{t_k}\frac{\|x^{k+1}-x^k\|^2}{c(\Phi(y^k)-\Phi^*)^{\alpha}},& k\in J,   
    \end{cases}                               \]
    where for the last inequalities we have used the Armijo conditions \eqref{eq:hybrid_armijo_rgd} and \eqref{eq:hybrid_armijo_pgd}. Using \cref{connect_riemannian_minnorm}, we have
   \[  \forall k\notin J,\quad \tau^k -\tau^{k+1}\geq \eta\|y^{k+1}-y^k\|\geq \frac{\eta}{Q}\|x^{k+1}-x^k\|,    \]
    where we set $Q$ to be the Lipschitz modulus of $T$ on the compact set $T^{-1}(K)$. Now we define $J_1:=\{k\in J:~\phi(x^{k+1})-\phi(x^*)\leq (\phi(x^k)-\phi(x^*))/2\}=\{j_{k}:~k\in \n\}$, and $J_2:=J\setminus J_1$. Assume that $J_1$ is infinite. For $j_k\in J_1$, we have 
    \begin{align*}
      \frac{1}{2^{k}}(\phi(x^{0})-\Phi^*) &\geq \frac{1}{2^{k}}(\phi(x^{j_0})-\Phi^*) \geq\phi(x^{j_k})-\Phi^* \geq \phi(x^{j_k})-\phi(x^{j_k+1})\\
      &\geq  \frac{\eta}{t_{j_k}}\|x^{{j_k}+1}-x^{j_k}\|^2\geq  \frac{\eta}{t_0}\|x^{j_k+1}-x^{j_k}\|^2.
    \end{align*} 
    In particular, this means that 
   \[  \forall k\in \n,~ \sum_{k=0}^\infty\|x^{j_k+1}-x^{j_k}\|\leq  \sqrt{\frac{t_0(\phi(x^0)-\Phi^*)}{\eta}}\sum_{k=0}^\infty \frac{1}{2^{\frac{k}{2}}} <\infty.         \]
    For $k\in J_2$, we have 
    \begin{align*}
      \quad \tau^k -\tau^{k+1}\geq \frac{\eta}{t_k}\frac{2^{{-\alpha}}\|x^{k+1}-x^k\|^2}{c(\phi(x^{k+1})-\Phi^*)^{\alpha}}\geq \frac{\eta}{t_0}\frac{2^{{-\alpha}}\|x^{k+1}-x^k\|^2}{\dist(0,\partial \phi(x^{k+1}))}.
    \end{align*}
    Utilizing the optimality condition of the projected gradient step and the Lipschitz continuity of $\nabla f$ on $K$, we have
    \begin{align}
       &\frac{1}{t_k}(x^k-x^{k+1})-\nabla f(x^k) \in N_{\cX}(x^{k+1}), \notag\\ 
       & \nabla f(x^{k+1})+\frac{1}{t_k}(x^k-x^{k+1})-\nabla f(x^k)\in \partial \phi(x^{k+1}) \notag\\
       & \dist(0,\partial \phi(x^{k+1}))\leq \|\nabla f(x^{k+1})-\nabla f(x^k)\|+\frac{1}{t_k}\|x^{k+1}-x^k\|\notag\\
       &\hspace{7.3em}\leq \left(L+\frac{1}{\underline{t}}\right)\|x^{k+1}-x^k\|,\label{eb_dist_pgd}
    \end{align}
    where $L$ is the Lipschitz modulus of $\nabla f$ on $K$. The case where $J_1$ is finite can be analyzed similarly, and we omit the details. Therefore, we obtain that $\{x^k\}_k$ has finite length by summing up the above inequalities, which further implies the convergence of the whole sequence $\{x^k\}_k$ to $x^*$.
\end{proof} 
While \cref{convergence_1} ensures convergence to a stationary point, it does not give a convergence rate. This is also explained in \cite{tang2024optimization}: the Riemannian gradient step cannot guarantee sufficient decrease in the function value, which means that the theoretical convergence rate will depend on the number of projected gradient steps. We now carefully analyze the manifold structure and show sufficient decrease for the Riemannian gradient step under suitable conditions. This allows us to establish a convergence rate for \cref{alg:hybrid_first_order} by using our KL results. {The proof consists of two steps: first, we establish a uniform lower bound on the Riemannian gradient stepsizes by bounding a suitable choice of $\lambda$ in \eqref{rg}; second, we combine the decrease estimates for both types of steps with our KL results to establish the convergence rate.}

To achieve this, we need a specific retraction. For each $I\subseteq [n]$, we define the retraction $R_y$ on $\cY_I$ as
\begin{equation}
  \label{retra_1}
 \forall (y\circ v)\in T_{\cY_I}(y),\quad  R_y(y\circ v)= y\circ \sqrt{\mathbf{1}+2v}
\end{equation}
This is only defined for sufficiently small $v$, but is sufficient to the convergence analysis, as indicated by \cite[{Definition 1}]{absil2012projection}. We define the retraction for $y\circ v$ rather than for $v$ since the Riemannian gradient takes the form $y\circ v$ as in \cref{rg}. Indeed, it is easy to verify that {$R_y$} defined in \cref{retra_1} satisfies the retraction properties at any $y\in \cY_I$. Using this retraction, let us establish {the lower bound on the Riemannian gradient stepsizes}. For any $y\in \cY_I$, set $x=y^2$ and select $\lambda$ such that $2y\circ (\nabla f(x)+A^\top\lambda)\in T_{\cY_I}(y)$. Using the descent lemma, we have 
\begin{align*}
   &\Phi(R_y(-2\alpha y\circ (\nabla f(x)+A^\top\lambda)))=f(R_y(-2\alpha y\circ (\nabla f(x)+A^\top\lambda))^2)\\
   &=f(x\circ(\mathbf{1}-4\alpha (\nabla f(x)+A^\top\lambda)))\\
   &\overset{\rm (a)}{\leq} f(x)-4\alpha\langle \nabla f(x), x\circ (\nabla f(x)+A^\top\lambda) \rangle +8L\alpha^2\|x\circ (\nabla f(x)+A^\top\lambda)\|^2 \\
   &\overset{\rm (b)}{=}  f(x)-4\alpha\langle \nabla f(x)+A^\top\lambda, x\circ (\nabla f(x)+A^\top\lambda) \rangle +8L\alpha^2\|x\circ (\nabla f(x)+A^\top\lambda)\|^2  \\
   &\leq f(x)-\alpha \|2y\circ(\nabla f(x)+A^\top\lambda)\|^2+2L\alpha^2\|y\|_\infty^2\|2y\circ (\nabla f(x)+A^\top\lambda)\|^2\\
   &=f(x)-\alpha(1-2L\alpha\|y\|_\infty^2)\|\grad_{\cY_I}\Phi(y)\|^2,
\end{align*}
where in (a) we have used the Lipschitz continuity of $\nabla f$ on a sufficiently large compact set with modulus $L$, and in (b) we have used the fact that $\langle A^\top\lambda, x\circ (\nabla f(x)+A^\top\lambda)\rangle=0$ since \cref{yi_tn_cones} gives $y\circ y\circ (\nabla f(x)+A^\top\lambda) = x\circ (\nabla f(x)+A^\top\lambda ) \in \ker(A)$. On the other hand, since for all $t\in [-0.5,\infty)$ it holds that $|\sqrt{1+2t}-1|=2|t|/(\sqrt{1+2t}+1)$, for all $\epsilon>0$ there exists sufficiently small $\delta>0$ such that if $\|v\|_\infty<\delta$, then 
$$  \|y\circ \sqrt{\mathbf{1}+2v}-y\|\leq (1+\epsilon)\|y\circ v\|. $$
In particular, for any $\kappa>0$, there exists $\overline{\alpha}>0$, such that when $\|\nabla f(x)+A^\top \lambda\|,\|y\|\leq \kappa$ and $\alpha\leq \overline{\alpha}$, then 
\begin{equation}
  \label{sufficient_decrease}
  \begin{aligned}
    &\Phi(R_y(-2\alpha y\circ (\nabla f(x)+A^\top\lambda)))\\
    &\leq \Phi(y)-\eta\|\grad_{\cY_I}\Phi(y)\|\|R_y(-2\alpha y\circ (\nabla f(x)+A^\top\lambda))-y\|.
  \end{aligned}
\end{equation}
Recall in the definition of Riemannian gradient \cref{rg}, we have 
\[  \lambda \in  \argmin_{\rho\in \R^m}\|2y\circ (\nabla f(x)+A^\top\rho)\|,    \]
and in particular $\lambda$ can be taken to be $\lambda= {-(A\diag(x)A^\top)^\dagger(A(x\circ \nabla f(x)))}$. By \cite[Theorem 1]{stewart1989scaled}, there exists a constant $\rho_A>0$ depending only on $A$ such that\footnote{In \cite[Theorem 1]{stewart1989scaled}, it is required that $A$ has full row rank and $y$ is positive. The assumption that $A$ has full row rank is not necessary as pointed out in \cite[Remark 1]{wei1995upper}. The assumption that $y$ is positive can be removed by considering the submatrix of $A$ corresponding to the {nonzero} entries of $y$.}:
\begin{align}
  \label{bound_lambda}
    \| A^\top\lambda\|\leq \|A\|_2\|\lambda\|\leq {\rho_A}\|A\|_2\|\nabla f(x)\|.
\end{align}
Therefore, for each bounded set $\cB\subseteq \R^n$, there exists a constant $\kappa>0$ such that for all $y\in T^{-1}({\cB\cap\cX})$ and $I=\{i\in [n]:~y_i\neq 0\}$, it holds that 
\[    \grad_{\cY_I}\Phi(y)=2y\circ (\nabla f(y^2)+A^\top\lambda),~\|\nabla f(y^2)+A^\top\lambda\|\leq \kappa.      \]
Combining this with the sufficient decrease in \cref{sufficient_decrease}, we know if the sequence $\{y^k\}_k$ generated by \cref{alg:hybrid_first_order} is bounded, then there exists $\underline{\alpha}$ such that $\alpha_k\geq \underline{\alpha}$ for all $k\notin J$. In addition, using the bound $|\sqrt{1+2t}-1|=2|t|/(\sqrt{1+2t}+1)$ for all $t\geq -0.5$ again, we have 
\begin{equation}
  \label{lowerbound_rgd}
  \begin{aligned}
    &\|R_{y}(-2\alpha y\circ (\nabla f(x)+A^\top\lambda))-y\| \geq \frac{2\alpha}{\sqrt{1+2\kappa}+1} \|2y\circ (\nabla f(x)+A^\top\lambda)\| \\
   &= \frac{2\alpha}{\sqrt{1+2\kappa}+1} \|\grad_{\cY_I}\Phi(y)\|=\frac{2\alpha}{\sqrt{1+2\kappa}+1} \dist(0,\partial \Phi(y))
  \end{aligned}
\end{equation}
provided that $2\alpha \|\nabla f(x)+A^\top\lambda\|_\infty\leq \kappa$.

\begin{prop}
  \label{convergence_rate_hybrid}
 Consider the optimization problem \cref{sqprob1} with $g=\iota_{\cX}$. Let $\{y^{k}\}_k$ be generated by \cref{alg:hybrid_first_order}, and assume $f$ is subanalytic. Assume that the retraction $R_y$ is given in \cref{retra_1}. If $\{y^k\}_k$ is bounded, then $\{x^k\}_k$ converges to a stationary point $x^*$ of $\phi$ in \cref{prob1}. If furthermore $\phi$ has KL exponent $\alpha$ at $x^*$, then the following statements hold:
  \begin{enumerate}[label=\rm (\roman{*})]
    \item  If $-\nabla f\left(x^*\right) \in \operatorname{ri}(\partial g(x^*))$, then $\left\|x^k-x^*\right\|=O\left(k^{-\frac{1-\alpha}{2 \alpha-1}}\right)$ if $\alpha \in\left(\frac{1}{2}, 1\right)$, and $\left\|x^k-x^*\right\|=O\left(c^k\right)$ for some $c \in(0,1)$ if $\alpha \in\left(0, \frac{1}{2}\right]$.
    \item  If $\phi$ is convex and the conditions in \cref{lemma5-1} are all satisfied, then it holds that $\left\|x^k-x^*\right\|=O\left(k^{-\frac{1-\beta}{2 \beta}}\right)$, where $\beta$ is defined in \cref{lemma5-1}.
  \end{enumerate} 
\end{prop}
\begin{proof}
   By \cref{convergence_1}, we know $x^k\to x^*$. By the Armijo rule in \cref{eq:hybrid_armijo_pgd} and the bound in \cref{eb_dist_pgd}, we have 
    \[ \forall k\in J,~~\phi(x^k)-\phi(x^{k+1})\geq \frac{\eta}{t_k}\|x^{k+1}-x^k\|^2\geq \frac{\eta}{t_0(L+1/\underline{t})^2}\dist(0,\partial \phi(x^{k+1}))^2.        \]
   Let $\tau^k:=\phi(x^k)-\phi(x^*)$. Since $\phi(x^k)\downarrow\phi(x^*)$, for sufficiently large $k$, we can apply the KL inequality at $x^*$ to show that for all sufficiently large $k\in J$ it holds that 
    \[ \phi(x^k)-\phi(x^{k+1})=\tau^k-\tau^{k+1}\geq a_1(\tau^{k+1})^{2\alpha},  \]
    where $a_1>0$ is a constant independent of $k$. For $k\notin J$, by the Armijo rule in \cref{eq:hybrid_armijo_rgd} and the bound in \cref{sufficient_decrease}, we have
   \begin{align*}
       \tau^k-\tau^{k+1}=\Phi(y^k)-\Phi(y^{k+1})\geq \eta \dist(0,\partial\Phi(y^k))\|y^{k+1}-y^k\|.
   \end{align*}
   Using the bound in \cref{lowerbound_rgd}, we have for all sufficiently large $k\notin J$ 
   \[   \tau^k-\tau^{k+1}    \geq a_2 \dist(0,\partial\Phi(y^k))^{2},          \]
   where $a_2>0$ is some constant independent of $k$. Let us note that the Riemannian gradient would not change the sign of $y^k$ and {that} for each $k\in J$, we have $y^{k+1}=\sqrt{x^{k+1}}$, {so} we know that $y^{k}$ is nonnegative for all sufficiently large $k$. Since $x^k\to x^*$, we also have $y^k=\sqrt{x^k}\to y^*:=\sqrt{x^*}$. Using the KL property of $\Phi$ at $y^*:=\sqrt{x^*}$, we have for all sufficiently large $k\notin J$ that
   \[   \tau^k-\tau^{k+1}    \geq a_2 (\tau^k)^{2\gamma}\geq a_2(\tau^{k+1})^{2\gamma},          \]
   where $\gamma={\max\{\alpha,1/2\}}$ if $-\nabla f\left(x^*\right) \in \operatorname{ri}(\partial g(x^*))$ by \cref{thmstrict}, and $\gamma={(1+\beta)/2}$ if $\phi$ is convex and the conditions in \cref{lemma5-1} are all satisfied by {\cref{thm-KLnonstrict}}. The rate analysis then follows from a similar analysis as in \cite[{Theorem 3.4}]{frankel2015splitting}.

\end{proof}
}

\bibliographystyle{siamplain}
\bibliography{Commonbib}
\end{document}